\documentclass[11pt,a4paper]{article}
\usepackage{amsmath,amsthm,amssymb}
\usepackage{graphicx}
\usepackage{subfigure}
\usepackage{multirow}
\newtheorem{theorem}{Theorem}[section]
\newtheorem{lemma}{Lemma}[section]

\newtheorem{remark}{Remark}[section]

\newcommand{\se}{\setcounter{equation}{0}}

\newcommand{\cT}{\mathcal{T}}

\newcommand{\vhs}{V_h^*}

\newcommand{\cL}{A}

\def \R{{{\rm I{\!}\rm R}}}
\newcommand{\Ba}{\partial_t^{1-\alpha}}
\newcommand{\I}{\mathcal{I}}
\newcommand{\J}{\mathcal{J}}

\newcommand{\Bac}{{^C}\partial_t^{\alpha}}

\newcommand{\norm}[1]{{\left\vert\kern-0.25ex\left\vert\kern-0.25ex\left\vert #1 
    \right\vert\kern-0.25ex\right\vert\kern-0.25ex\right\vert}}

\title{Error analysis of a finite volume element method for fractional order evolution equations 
with nonsmooth initial data} 
\author{Samir Karaa{\thanks{ Department of Mathematics and Statistics, Sultan Qaboos University,
 P. O. Box 36, Al-Khod 123,
 Muscat, Oman. Email: skaraa@squ.edu.om. The research of this author is supported by Sultan Qaboos University under Grant
 IG/SCI/DOMS/16/01.} }
 and
 Amiya K. Pani{\footnote{
 Department of Mathematics, Industrial Mathematics Group, Indian
 Institute of Technology Bombay,
 Powai, Mumbai-400076.
 }}}
\begin{document}
\date{}
\maketitle
\begin{abstract}
{In this paper,
a finite volume element (FVE) method is considered for spatial approximations of time-fractional diffusion 
equations involving a Riemann-Liouville fractional derivative of order $\alpha \in (0,1)$ in time. 
Improving upon earlier results (Karaa {\it et al.}, IMA J. Numer. Anal. 2016), 
optimal 
error estimates  in $L^2(\Omega)$- and $H^1(\Omega)$-norms 
for the  semidiscrete problem 
with smooth and middly smooth initial data, i.e., $v\in H^2(\Omega)\cap  H^1_0(\Omega)$ 
and $v\in H^1_0(\Omega)$ are established. For nonsmooth  data, that is, $v\in L^2(\Omega)$, 
the optimal $L^2(\Omega)$-error estimate
is shown to hold only under an additional assumption on the triangulation, which is known to be 
satisfied for symmetric triangulations. Superconvergence result is also proved and as a consequence,
 a quasi-optimal error estimate is established in the $L^\infty(\Omega)$-norm.
Further,  two fully discrete schemes using convolution quadrature in time generated by the 
backward Euler and the second-order backward difference methods are analyzed, and error estimates are 
derived for both smooth and nonsmooth initial data. 
Based on a comparison  of the standard Galerkin finite element
 solution with the FVE solution and exploiting tools for Laplace transforms with semigroup type properties of the 
FVE solution operator, our analysis is then extended in a unified manner to several 
time-fractional order evolution problems. 
Finally, several  numerical experiments are  conducted
to confirm  our theoretical findings.}
\end{abstract}



\section{Introduction}
\noindent
Let $\Omega$ be a bounded, convex polygonal domain in $\mathbb{R}^2$  with
boundary $\partial \Omega$, $T>0,$  and let $v$ be a given function (initial data) defined on $\Omega$. 
We now consider the following  time-fractional diffusion problem: find $u$  in 
$ \Omega\times (0,T] $ such that
\begin{subequations}\label{main}
\begin{alignat}{2}\label{a1}
&u'(x,t)+\Ba \cL u(x,t)=0 &&\quad\mbox{ in }\Omega\times (0,T],
\\  \label{a2}
&u(x,t)= 0 &&\quad\mbox{ on }\partial\Omega\times (0,T],
\\   \label{a3}
&u(x,0)=v(x) &&\quad\mbox{ in }\Omega,
\end{alignat}
\end{subequations}
where $\cL u=-\Delta u$,  $u'$ is the partial derivative of
  $u$ with respect to time, and $\Ba:={^R}{\rm D}^{1-\alpha}$ is the Riemann-Liouville  fractional
derivative in time defined for $0<\alpha<1$  by:  
\begin{equation} \label{Ba}
\Ba \varphi(t):=\frac{d }{d t}\I^{\alpha}\varphi(t):=\frac{d }{d t}\int_0^t\omega_{\alpha}(t-s)\varphi(s)\,ds\quad\text{with} \quad
\omega_{\alpha}(t):=\frac{t^{\alpha-1}}{\Gamma(\alpha)}.
\end{equation}
Here,  $\I^{\alpha}$ denotes the temporal Riemann-Liouville fractional integral operator of order $\alpha$. 
This class of problems describes the model of  an anomalous subdiffusion, see \cite{GMMP}, \cite{HW}
and  \cite{MK}.

Over the last two decades, 
considerable attention 
from both practical and theoretical point of views has been given  to fractional  diffusion models due to
their various applications. Several numerical techniques  for the 
 problem \eqref{main}  have been proposed with different types of spatial discretizations.
The finite element (FE) method has, in particular,  been given a special attention  in approximating the solution 
of the problem \eqref{main},  see \cite{ McLeanThomee2004,MT2010,McLeanThomee2010, MustaphaMcLean2011, 
JLZ2013,JLPZ2015, JLZ2016,EJLZ2016} and references, there in.  Most recently, a 
FVE method is analyzed in \cite{KMP2015} and  {\it a prior} error estimates 
with respect to data regularity  have been derived.

Although the numerical study of  \eqref{main} has been  discussed in a large number of papers,
optimal error estimates with respect to the smoothness of the solution  expressed through initial data
have been established only in few papers recently. 
This is due to the presence of time-fractional derivative, and hence, deriving  sharp error bounds under 
reasonable regularity assumptions on the exact solution has become a  challenging task.

To motivate our results, we begin by recalling some facts on the spatially semidiscrete standard 
Galerkin FE method for the problem \eqref{main} in the piecewise FE element space
$$V_h=\{\chi\in C^0(\overline {\Omega})\;:\;\chi|_{K}\;\mbox{is linear for all}~ K\in \cT_h\; 
\mbox{and} \; \chi|_{\partial \Omega}=0\},$$
where $\{\cT_h\}_{0<h<1}$ is a family of regular  triangulations $\cT_h$ of
the domain $\Omega$ into triangles  $K$  with  $h$ denoting the maximum diameter  of the triangles $K\in \cT_h$.  
With $a(\cdot,\cdot)$ denoting the bilinear form associated with the operator $A$, and $(\cdot,\cdot)$ the inner product in
$L^2(\Omega)$, the semidiscrete Galerkin FE method is  to seek  $u_h(t)\in V_h$ satisfying
\begin{equation} \label{semi-FE}
( u_h',\chi)+ a(\Ba u_h,\chi)=  0\quad
\forall \chi\in V_h,\quad t\in (0,T], \quad u_h(0)=v_h,
\end{equation}
where $a(v,w):= (\nabla v, \nabla w)$ and $v_h\in V_h$  is an approximation of the initial data $v$. Upon introducing  the discrete 
operator $A_h:V_h\rightarrow V_h$ defined by
$$
(A_h\psi,\chi)=(\nabla \psi,\nabla \chi)  \quad \forall \psi,\chi\in V_h,
$$
the semidiscrete FE scheme \eqref{semi-FE}  is  rewritten in an operator form as
\begin{equation} \label{FEP}
 u_h'(t)+\Ba A_h  u_h(t)=  0, \quad t>0, \quad  u_h(0)=v_h.
\end{equation}
In  \cite{MT2010}, McLean  and Thom\'ee have established the following estimate for the Galerkin FE 
approximation to  \eqref{main}: with $v_h=P_hv$, there holds for $t>0$
\begin{equation} \label{FE-1}
\|u_h(t)-u(t)\|\leq Ch^2t^{-\alpha(2-q)/2}|v|_{q}, \quad 0\leq q\leq 2,
\end{equation}
where $\|v\|$ is the $L^2(\Omega)$-norm of $v$ and  $|v|_{q}=\|A^{q/2}v\|$ is a weighted norm defined on the space 
$\dot{H}^q(\Omega)$ to be  described in 
Section 2. Here, $P_h:L^2(\Omega)\rightarrow V_h$ is the $L^2$-projection given by : 
$(P_hv-v,\chi)=0$ for all $\chi\in V_h$. For a smooth initial data, that is, $v\in \dot{H}^2(\Omega)$, 
the estimate \eqref{FE-1} is still valid for the initial approximation $v_h=R_hv,$ where  
$R_h:H_0^1(\Omega)\rightarrow V_h$ is the standard Ritz projection defined by the relation: 
$a(R_hv-v,\chi)=0$ for all $\chi\in V_h$.
The estimate \eqref{FE-1} extends  results obtained for the standard parabolic problem, i.e, $\alpha=1$, 
which has been thoroughly studied, see  \cite{thomee1997}. In the recent work \cite{EJLZ2016}, 
an approach based on Laplace transform and semigroup type theory
 has been exploited to derive {\it a priori} error estimates of the type \eqref{FE-1}, and 
most recently, a delicate energy analysis has been developed in \cite{KMP2016}  to obtain similar
estimates.

Regarding the optimal estimate in the gradient norm,  the following result
\begin{equation} \label{FE-2}
\|\nabla(u_h(t)-u(t))\|\leq Cht^{-\alpha(2-q)/2}|v|_{q}, \quad 0\leq q\leq 2,
\end{equation}
holds with $v_h=P_hv$ on quasi-uniform meshes. For the cases $q=1,2$, one can also choose $v_h=R_hv$. 
However, without the quasi-uniformity  assumption on the mesh, the estimate \eqref{FE-2} remains valid only for $0\leq q\leq 1$, see \cite{KMP2016}.

Optimal convergence rate up to a logarithmic factor in the stronger $L^\infty(\Omega)$-norm has 
been derived in \cite{McLeanThomee2010,KMP2016}. While in \cite{McLeanThomee2010}, Laplace transform
technique combined with semigroup type theoretic approach is used to derive maximum norm estimates,
in \cite{KMP2016}  a novel
energy argument combined with Sobolev inequality for $2$D-problems is employed 
to establish, under quasi-uniformity assumption on the mesh, the following $L^{\infty}(\Omega)$-error estimate 
for $v \in \dot H^q(\Omega)\cap L^\infty(\Omega)$
and $v_h=P_h v$   
\begin{equation} \label{FE-3}
 \|u(t)-u_h(t)\|_{L^{\infty}(\Omega)} \leq
 C|\ln h|^{\frac{5}{2}} h^2 t^{-\alpha(3-q)/2}(|v|_q+\|v\|_{L^\infty(\Omega)}), \quad 1\le q\le 2.
\end{equation}

In this article, we discuss the error analysis of the approximate solution $\bar{u}_h$  
satisfying the following FVE  method:
\begin{equation} \label{semi-FV}
(\bar{u}_h',\chi)_h+ a(\Ba \bar{u}_h,\chi)=  0\quad
\forall \chi\in V_h,\quad t\in (0,T], \quad \bar{u}_h(0)=v_h,
\end{equation}
where $(\cdot,\cdot)_h$ is a discrete inner product  on $V_h$ to be defined in  Section \ref{sec: FVM}. Here, one of our objective is to 
establish the analogous of estimates
(\ref{FE-1}) and (\ref{FE-2})  for the solution of the FVE semidiscrete problem \eqref{semi-FV}, namely;
 with the appropriate choices of $v_h$,  
\begin{equation} \label{FVE}
\|{\bar u}_h(t)-u(t)\|+h \|\nabla({\bar u}_h(t)-u(t))\|\leq Ch^2t^{-\alpha(2-q)/2}|v|_{q},  \quad 0\leq q\leq 2.
\end{equation}
We shall derive this  estimate for $q=1,\,2$ in Section \ref{sec:H2} and 
for $q=0$ in Section \ref{sec:L2}. For the latter case, we are only able to prove the {\it a priori} 
estimate under an additional hypothesis on ${\mathcal T}_h$, which is known to be satisfied for 
symmetric triangulations. 
Without any such condition,  only sub-optimal order convergence is  obtained, which is similar to the 
result proved in \cite{CLT-2013} for linear parabolic problems. For the stronger $L^\infty(\Omega)$-norm, 
a quasi-optimal error estimate analogous to (\ref{FE-3}) is established for $1\leq q\leq 2$.

Our analysis provides improvements of earlier results in \cite{KMP2015}, where the initial data $v$ is 
required to be in $\dot H^q(\Omega)$ with $q\geq 3$. Unlike  the classical FE error analysis
 in which  an intermediate projection, usually, a Ritz projection, is introduced to derive optimal error 
 estimates, our approach, here, shall combine the error estimates for the standard Galerkin FE
 solution stated above with new bounds for the difference $\xi(t)=\bar u_h(t)-u_h(t)$. A similar idea  
 has been used in \cite{CLT-2012} and 
\cite{CLT-2013} for the approximation of the standard parabolic problem by the lumped mass FE method 
and the  FVE method, respectively, leading to an improvement of their earlier results 
in \cite{CLT-2004}.

Our second objective is to analyze two fully discrete schemes for the semidiscrete problem 
\eqref{semi-FV} based on convolution quadrature in time generated by the backward Euler and 
the second-order backward difference methods. Error estimates with respect to the data regularity 
are provided in Theorems \ref{thm:BE} and \ref{thm:SBD}. For instance, it is shown that the discrete 
solution $U_h^n$ obtained by the backward Euler method 
with a time step size $\tau$ satisfies the following {\it a priori} error estimate 
$$
\|U_h^n-{\bar u}_h(t_n)\|\leq C(\tau^{-1+\alpha q/2}+ h^2t_n^{-\alpha(1-q/2)})|v|_{q},  \quad q=0,1,2.
$$
When $q=0,$ an additional restriction on the triangulation is imposed. A similar type of error bound 
is shown to hold for the second-order backward difference scheme in Subsection 5.2.

Our third objective is to generalize our results on FVE method for both smooth and nonsmooth initial data
to other classes of fractional order evolution equations in Section 6. Say for example,  we can extend our
FVE analysis to  the following 
class of time-fractional problems: 
\begin{equation}\label{J-alpha-class}
u'(x,t)+ \J^{\alpha} \cL u(x,t)=0 \quad\mbox{ in }\Omega\times (0,T],
\end{equation}
with homogeneous Dirichlet boundary conditions and initial condition $u(x,0)= v(x)$ for $x\in \Omega.$
When $\J^{\alpha}= \I^{\alpha},$ this class of problems is known as fractional diffusion-wave equation or 
evolution equation with positive memory, see \cite{LST-1996, MT2010} and references, therein.  The case $\J^{\alpha}= 
I+\I^{\alpha}$ corresponds to the PIDE with singular kernel, refer to \cite{MST2006}. Now if $\J^{\alpha} = I + \Ba,$ 
then this class of
problems is known as the Rayleigh-Stokes problems for generalized second grade fluid, see \cite{EJLZ2016}.   
Even our FVE analysis can be directly  applied  to the following time-fractional order diffusion problem:
\begin{equation}\label{caputo}
\Bac u(x,t) 
+ \cL u(x,t)= 0,
\end{equation}
where $\Bac v(t):=\I^{1-\alpha}v'(t)$ is the  fractional Caputo derivative of order $0<\alpha<1.$
For the semidiscrete FE analysis of (\ref{caputo}), we refer to 
Jin  {\it et al.} \cite{JLZ2013}. The unifying analysis of all these classes of evolution problems is based on 
 comparing the FVE solution with the corresponding FE solution and exploiting the Laplace transform technique
 along with semigroup type properties  of the FVE solution operator.


The rest of the paper is organized as follows. In the next section, we introduce notation, 
recall the solution representation for the continuous problem \eqref{main} and some smoothing properties of the solution operator, which  play an important role in our subsequent error analysis. Section 3 deals with
 a brief description of 
the spatially semidiscrete  FVE scheme and their properties. In Section \ref{sec:error},
 we  derive error estimates for the semidiscrete FVE scheme for smooth and nonsmooth initial data 
 $v\in \dot H^q$, $q=0,1,2$ in Subsections \ref{sec:H2} and \ref{sec:L2}. For  $q=0$, i.e., $v\in L^2(\Omega)$, 
 we show an optimal error bound under 
 an additional  assumption on the triangulation. Superconveregence result is proved in 
 Subsection \ref{sec:Linfty} and as a consequence, a quasi-optimal error estimate is established in the 
$L^\infty(\Omega)$-norm.
In Section \ref{sec:discrete}, two fully discrete schemes  based on convolution quadrature 
approximation of the fractional derivative are presented and error estimates are established. 
Section \ref{sec:extensions} focuses on possible generalization of the present FVE error analysis to 
various types of time-fractional  evolution problems.  Finally, 
in Section \ref{sec:NE}, we present  numerical results  to confirm  our theoretical findings. 

Throughout the paper, $C$ denotes a generic positive constant that may depend on $\alpha$ and $T$, 
but is independent of the spatial mesh element size $h$.

\section{Representation of exact solution and properties } \label{sec:notation}
\se
We first introduce some notations.  Let $\{(\lambda_j,\phi_j)\}_{j=1}^\infty$ be the Dirichlet eigenpairs 
of  the selfadjoint and positive definite operator $\cL$, with $\{\phi_j\}_{j=1}^\infty$ being an 
 orthonormal basis in $L^2(\Omega)$.
For  $r\geq 0$, we denote by  $\dot H^r(\Omega)\subset L^2(\Omega)$  the Hilbert space 
induced by the norm
 \[ |v|_r^2 =\|\cL^{r/2}v\|^2 =\sum_{j=1}^\infty \lambda_j^r (v,\phi_j)^2,\]
 with $(\cdot,\cdot)$ being the inner product on $L^2(\Omega)$.
Then,  it follows that 
$
\dot{H}^r(\Omega)=\{\chi \in H^r(\Omega);\, \cL^j\chi=0 \text{ on } \partial \Omega, \text{ for } j<s/2\}$, 
see \cite[Lemma 3.1]{thomee1997}. In particular, $|v|_0=\|v\|$ is the norm on $L^2(\Omega)$, 
$|v|_1=\|\nabla v\|$ is also the norm on $H_0^1(\Omega)$ and $|v|_2=\|A v\|$ is the equivalent 
norm in $H^2(\Omega)\cap H^1_0(\Omega)$. Note that $\{\dot H^r(\Omega)\}$, $r\geq 0$,  form a Hilbert scale of interpolation spaces. Motivated by this, we denote  by $\|\cdot\|_{H^r_0(\Omega)}$ the norm on the interpolation scale between $H^2(\Omega)\cap H^1_0(\Omega)$ and $L^2(\Omega)$ for $r$ in the interval $[0,2]$. 
Then, the $\dot H^r(\Omega)$ and $H^r_0(\Omega)$ norms are equivalent for any $r\in (1/2,2]$  for $r\in [0,1/2],$
$\dot H^r(\Omega)= H^r(\Omega)$ by interpolation.

For $\delta>0$ and $\theta\in (\pi/2,\pi)$, we introduce the contour $\Gamma_{\theta,\delta}\subset \mathbb{C}$ defined by
$$
\Gamma_{\theta,\delta}=\{\rho e^{\pm i\theta}:\rho\geq \delta\}\cup\{\delta e^{i\psi}: |\psi|\leq \theta\},
$$
oriented with an increasing imaginary part. Further, we denote by $\Sigma_{\theta}$ the sector
$$
\Sigma_{\theta}=\{z\in \mathbb{C}, \,z\neq 0,\, |\arg z|< \theta\}.
$$
For $z\in \Sigma_\theta$, it is clear that $z^\alpha \in \Sigma_\theta$ as $\alpha \in (0,1)$. 
Since the operator $A$ is selfadjoint and positive definite, its resolvent 
$(z^\alpha I+A)^{-1}:L^2(\Omega)\rightarrow L^2(\Omega)$ satisfies the bound
\begin{equation}\label{res1}
\|(z^\alpha I+A)^{-1}\|\leq M_\theta |z|^{-\alpha} \quad \forall z\in \Sigma_\theta,
\end{equation}
where $M_\theta=1/\sin(\pi-\theta)$. 
We now make use of  the  Laplace transform $\hat{u}:=\mathcal{L}(u)$ of the solution $u$ defined by
$$
\hat{u}(z,x)=\int_0^\infty e^{-zt}u(t,x)\,dt.
$$
The boundary condition $u(x,t)=0$ on $\partial \Omega$ transforms into  $\hat u(x,z)=0$ 
on $\partial \Omega$. 
Taking Laplace transforms in (\ref{a1}), we, then, arrive at
\begin{equation}\label{m1}
(zI+z^{1-\alpha}A)\hat{u}(z)=v,
\end{equation}
and hence, 
\begin{equation}\label{m2}
\hat{u}(z)= \hat{E}(z)v, \quad \hat{E}(z)= z^{\alpha-1}(z^{\alpha}I+A)^{-1}.
\end{equation}
In view of  (\ref{res1}) and (\ref{m2}), $\hat{E}(z)$ satisfies the following bound
\begin{equation}\label{m3}
\|\hat{E}(z)\|\leq M_\theta |z|^{-1} \quad \forall z\in \Sigma_\theta.
\end{equation}
From \eqref{m2}, the Laplace inversion formula yields an integral representation 
for the solution of (\ref{main}) as 
\begin{equation}\label{mm}
u(t)=\frac{1}{2\pi i}\int_{\mathcal C} e^{zt}\hat{E}(z)v\,dz, \quad t>0,
\end{equation}
where the contour of integration ${\mathcal C}$, known as Bromwich contour,  is any line in the right-half plane parallel to the imaginary axis  and with Im$z$ increasing. Since $\hat{E}(z)$ is analytic in $\Sigma_\theta$ and satisfies the bound \eqref{m3}, the path of integration may, therefore, be deformed 
into the curve $\Gamma_{\theta,\delta}$ so that the integrand has an exponential decay property.

In the next lemma, we present some smoothing properties of the operator $\hat{E}(z)$ which play a key 
role in our error analysis. The estimates are proved for instance  in \cite[Lemma 2.2]{Lubich-2006}. 
Note that the first estimate  \eqref{00} given below is obtained by interpolation technique.
%
\begin{lemma}\label{lem:Ah} The following estimates hold:
\begin{equation}\label{00}
 \|A\hat{E}(z)\chi\|\leq C_{\theta} |z|^{\alpha(1-p/2)-1} |\chi|_p \quad \forall z\in\Sigma_\theta,\quad 0\leq p\leq 2, 
 \end{equation} 
 \begin{equation}\label{11}
 \|\nabla\hat{E}(z)\chi\|\leq C_{\theta} |z|^{\alpha/2-1} \|\chi\| \quad \forall z\in\Sigma_\theta,
 \end{equation} 
 where $C_{\theta}$ depends only on $\theta$.
\end{lemma}
In the next section, we introduce the semidiscrete finite volume element scheme.


\section {Semidiscrete FVE scheme and its properties}\label{sec: FVM}

To describe the finite volume element  formulation,  we first introduce the dual mesh on the domain $\Omega$.
Let $N_h$ be the set of nodes or vertices, that is, $$N_h :=\left\{P_i:P_i~~\mbox{ is a vertex of the element }~K \in
\cT_h~\mbox{and}~P_i\in \overline{\Omega}\right\}$$ and let
$N_h^0$ be the set of interior nodes in $\cT_h.$
Further, let $\cT_h^*$ be the dual mesh associated with the primary mesh $\cT_h,$ which is defined as follows. With $P_0$ as an interior node of the triangulation $\cT_h,$ let  $P_i\;(i=1,2\cdots
m)$ be its adjacent nodes (see, Figure ~\ref{fig:mesh} with $m=6$ ). Let $M_i,~i=1,2\cdots
m$ denote the midpoints of $\overline{P_0P_i}$ and let $Q_i,~i=1,2\cdots
m,$  be  the barycenters of the triangle $\triangle P_0P_iP_{i+1}$ with
$P_{m+1}=P_1$. The {\it control volume}
  $K_{P_0}^*$ is constructed  by joining successively $ M_1,~ Q_1,\cdots
  ,~ M_m,~ Q_m,~ M_1$. With $Q_i ~(i=1,2\cdots
m)$ as the nodes of $control~volume~$ $K^*_{p_i},$ let $N_h^*$ be the set of all dual nodes
$Q_i$. For a  boundary
node $P_1$,  the control volume $K_{P_1}^*$ is shown in Figure ~\ref{fig:mesh}. Note that the union
of the control volumes forms a partition $\cT_h^*$ of $\overline{\Omega}$.

\begin{figure}[htb]
 \begin{center}
  \includegraphics*[width=8.0cm,height=4.0cm]{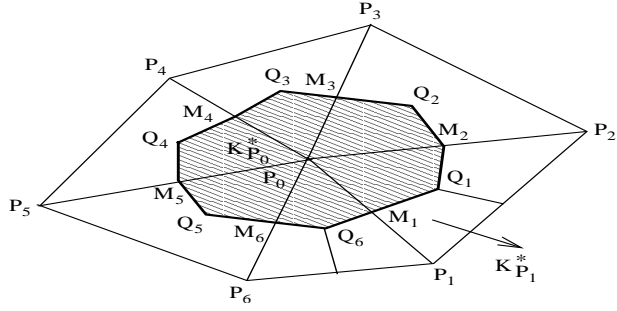}
    \caption{Control volume for interior node}
      \label{fig:mesh}
   \end{center}
\end{figure}

The dual volume element space $\vhs$ on the dual mesh $\cT^*_h$ is defined as 
$$\vhs=\{\chi\in L^2(\Omega)\;:\;\chi|_{K_{P_0}^*}\;\mbox{is constant for all}\; K_{P_0}^*\in \cT_h^*\; \mbox{and}\; \chi|_{\partial \Omega}=0\}.$$
The semidiscrete FVE  formulation for \eqref{main} is
to seek  $\bar u_h(t) \in  V_h$ such that
\begin{equation} \label{FV}
(\bar u_h',\chi)+ a_h(\Ba \bar u_h,\chi)=0 \quad
\forall \chi\in \vhs,\quad t>0, \quad \bar u_h(0)=v_h,
\end{equation}
where the bilinear form $a_h(\cdot, \cdot): V_h\times \vhs\longrightarrow \R$ is  defined by
\begin{equation} \label{Ah}
a_h(\psi,\chi)= -\sum_{P_i\in N_h^0} \chi(P_i) \int_{\partial K_{P_i}^*}
\nabla \psi\cdot{\bf n}\,ds\quad \forall \psi \in V_h,\; \chi\in  \vhs
\end{equation}
 with ${\bf n}$ denoting the outward unit normal to the boundary of the control volume
$K_{P_i}^*$.  For $w\in H^2(\Omega)$ and $\chi\in  \vhs$,  a use of Green's formula yields
\begin{equation*}\label{eq: h2 identity}
 (\cL w, \chi)= a_h(w,\chi).
\end{equation*}

To rewrite the Petrov-Galerkin method (\ref{FV}) as a Galerkin method in $V_h$, 
we introduce the interpolation operator
$\Pi_h^*:C^0(\bar{\Omega})\longrightarrow \vhs$  by
\begin{equation*}\label{naa}
\Pi_h^*\chi=\sum_{P_i\in N_h^0}\chi(P_i)\eta_i(x),
\end{equation*}
where $\eta_i$ is the characteristic function of the control volume $K_{P_i}^*$. The operator $\Pi_h^*$ 
is selfadjoint and positive definite, see \cite{ChouLi2000}, and hence, the following relation
\begin{equation*}\label{inner-1}
(\psi,\chi)_h=(\psi,\Pi_h^*\chi)\quad \forall \psi, \chi\in V_h
\end{equation*}
defines an inner product on $V_h$. Also, the corresponding norm  $(\chi,\chi)_h^{1/2}$ is equivalent 
to the $L^2(\Omega)$-norm on $V_h$, uniformly in $h$, 
see \cite{LiChenWu2000}. Furthermore, from the following identity \cite{BR-1987, EwingLazarovLin2000} 
$$
a_h(\chi,\Pi_h^* v)=(\nabla \chi,\nabla v)\quad \forall \chi, v\in V_h,
$$
the bilinear form $a_h(.,.)$ is symmetric and $a_h(\chi,\Pi_h^*\chi)=\|\nabla\chi\|^2$ for $\chi\in V_h$.

%
We now introduce the discrete operator $\bar A_h:V_h\rightarrow V_h$ corresponding to the inner product 
$(\cdot,\cdot)_h$ by
\begin{equation*} \label{Dh}
(\bar A_h\psi,\chi)_h=(\nabla \psi,\nabla \chi)  \quad \forall \psi,\chi\in V_h.
\end{equation*}
Then, the FVE method \eqref{semi-FV} is written in an operator form as
\begin{equation} \label{FVP}
\bar u_h'(t)+\Ba \bar A_h \bar u_h(t)=  0, \quad t>0, \quad \bar u_h(0)=v_h.
\end{equation}
%

An appropriate modification of arguments in \cite{CLT-2013,JLZ2013} yields the following
discrete analogous of Lemma \ref{lem:Ah} and therefore, we skip the proof.
\begin{lemma}\label{lem:Ah2} Let $\hat E_h(z)=z^{\alpha-1}(z^\alpha I+\bar A_h)^{-1}$. 
With $\chi\in V_h$, the following estimates hold:
\begin{equation}\label{0}
 \|{\bar A_h\hat{E}_h(z)\chi}\|\leq C_{\theta} |z|^{\alpha(1-p/2)-1} \;\|\bar {A}_h^{p/2}{\chi}\| 
 \quad \forall z\in \Sigma_\theta,\quad 0\leq p\leq 2,
 \end{equation} 
 \begin{equation}\label{1}
|{\hat{E}_h(z)\chi}|_1\leq C_{\theta} |z|^{\alpha/2-1} \|{\chi}\|\quad \forall z\in \Sigma_\theta,
 \end{equation}  
 where $C_{\theta}$ is independent of the mesh size $h$.
\end{lemma}
Moreover, an analogous of Lemma \ref{lem:Ah2} holds for $\hat F_h(z)=z^{\alpha-1}(z^\alpha I+A_h)^{-1}$, when 
we replace ${\hat{E}}_h(z)$ in Lemma \ref{lem:Ah2} by $\hat F_h(z).$


\section {Error analysis}\label{sec:error}
\se
This section deals with {\it a priori} optimal error estimates for the semidiscrete FVE scheme \eqref{semi-FV} with initial data $v\in \dot{H}^q(\Omega)$, $q=0,1,2$. To do so, 
we first introduce the quadrature error $Q_h:V_h\rightarrow V_h$ defined by
\begin{equation}\label{m8}
(\nabla Q_h\chi,\nabla\psi)=\epsilon_h(\chi,\psi):=(\chi,\psi)_h-(\chi,\psi)\quad \forall \psi \in V_h.
\end{equation}
The operator $Q_h$, introduced in \cite{CLT-2012} for the lumped mass FE element, represents the 
quadrature error in a special way. It satisfies the following error estimates, see \cite{CLT-2012,CLT-2013}.
\begin{lemma}\label{lem:Qh}
Let $Q_h$ be defined by \eqref{m8}. Then, there holds
\begin{equation}\label{QQ}
\|\nabla Q_h\chi\|+h\|\bar{A}_hQ_h\chi\|\leq Ch^{p+1}\|\nabla^{p}\chi\|\quad \forall \chi\in V_h, \quad p=0,1.
\end{equation}
\end{lemma}
Note that, by Lemma  \ref{lem:Qh}, and without additional assumptions on the mesh, the following estimate
holds:
$$\|Q_h\chi\| \leq C \|\nabla Q_h\chi\| \leq Ch \|\chi\|\quad \forall \chi\in V_h.$$
This estimate cannot be improved in general, see \cite{CLT-2012,CLT-2013} for some counter examples. 
However, on some special meshes, one can derive a better approximation.
For instance, if the mesh is symmetric (see \cite{CLT-2012,CLT-2013} for the definition and examples),  the operator $Q_h$ is shown to satisfy
\begin{equation}\label{sym}
\|Q_h\chi\|\leq Ch^2 \|\chi\| \quad \forall \chi\in V_h.
\end{equation}

To derive optimal error estimates for the FVE solution $\bar u_h$, we split the error 
$\bar{e}(t):=\bar{u}_h(t)-u(t)$ into 
$\bar{e}(t):=(u_h(t)-u(t))+\xi(t)$, where $\xi(t)=\bar{u}_h(t)-u_h(t)$ and $u_h$ being the standard Galerkin FE solution.
Then, from the definitions of $u_h(t)$, $\bar{u}_h(t)$ and $Q_h$, $\xi(t)$ satisfies
\begin{equation}\label{m10}
\xi_t(t)+\partial_t^{1-\alpha}\bar{A}_h\xi(t)=-\bar{A}_hQ_hu_{ht}(t),\quad  t>0,\quad  \xi(0)=0.
\end{equation}
%

\subsection {Error estimates for smooth initial data}\label{sec:H2}
In the following theorem,  optimal error estimates 
are derived for  smooth initial data $v\in \dot{H}^q(\Omega)$ with $q\in [1,2].$
\begin{theorem} \label{thm:H2} 
Let $u$ and $\bar{u}_h$  be the solutions of $\eqref{main}$ and $\eqref{semi-FV}$, respectively, 
with $v\in \dot{H}^q(\Omega)$ for $q\in [1,2]$ and $v_h=R_h v$. Then, there is a positive 
constant $C$, independent of
$h,$ such that 
\begin{equation}\label{e1}
\|\bar{u}_h(t)-u(t)\|+h\|\nabla(\bar{u}_h(t)-u(t))\|\leq C\;t^{-\alpha(2-q)/2}\;h^2|v|_q,\qquad  t>0.
\end{equation}
\end{theorem}
\begin{proof}
Since the estimates for $u_h-u$ are given in (\ref{FE-1}) and (\ref{FE-2}), it is sufficient to show  
\begin{equation}\label{0e2}
\|\xi(t)\|+h\|\nabla\xi(t)\|\leq C \;t^{-\alpha(2-q)/2}\; h^2 |v|_q,\;\; q\in [1,2].
 \end{equation}
 By taking Laplace transforms in \eqref{m10} and following the analysis in Section \ref{sec:notation}, we represent $\xi(t)$  by
\begin{equation}\label{m013}
\xi(t)=-\frac{1}{2\pi i}\int_{\Gamma} e^{zt}\hat{E}_h(z)\bar{A}_hQ_h\widehat{u_{ht}}(z)\,dz.
\end{equation} 
Here and also throughout this article, $\Gamma$ is the particular contour chosen as 
$\Gamma = \Gamma_{\theta,\delta}$ with $\delta=1/t$. From \eqref{m013}, it follows that
\begin{eqnarray} \label{xi-estimate-1}
\|\xi(t)\| + h \|\nabla \xi(t)\| \leq \frac{1}{2\pi}\int_{\Gamma} |e^{zt}| 
\Big(\|\hat{E}_h(z)\bar{A}_hQ_h\widehat{u_{ht}}(z)\| + h \|\nabla \hat{E}_h(z)\bar{A}_hQ_h\widehat{u_{ht}}(z)\| 
\Big) \,|dz|.
\end{eqnarray}
To complete the proof of the estimate, we need to compute the terms under the integral sign on the right 
of side of \eqref{xi-estimate-1}. Now, we discuss two cases for $q=2$ and $q=1$ separately.

When $q=2$, that is, $v\in \dot{H}^2(\Omega),$ apply (\ref{0}) with $p=1$ and (\ref{1}) in Lemma \ref{lem:Ah2} to obtain
\begin{equation}\label{m014}
\|\hat{E}_h(z)\bar{A}_hQ_h\widehat{u_{ht}}(z)\|\leq C |z|^{\alpha/2-1}\|\nabla Q_h \widehat{u_{ht}}(z)\|,
\end{equation}
and
\begin{equation}\label{g2}
\|\nabla \hat{E}_h(z)\bar{A}_hQ_h\widehat{u_{ht}}(z)\|\leq C |z|^{\alpha/2-1}\|\bar A_hQ_h\widehat{u_{ht}}(z)\|. 
\end{equation}
Then, by (\ref{QQ}), it follows that 
\begin{equation}\label{Q1aa}
\|\hat{E}_h(z)\bar{A}_hQ_h\widehat{u_{ht}}(z)\|+h\|\nabla \hat{E}_h(z)\bar{A}_hQ_h\widehat{u_{ht}}(z)\|\leq Ch^2 
|z|^{\alpha/2-1}\|\nabla \widehat{u_{ht}}(z)\|.
\end{equation}
Since 
$$\widehat{u_{ht}}(z)=-z^{1-\alpha}A_h\hat{{u}}_{h}(z)= -z^{1-\alpha}A_h\hat{F}_{h}(z)v_h,$$ 
an  estimate analogous to (\ref{1}) yields
\begin{equation}\label{Q1}
\|\nabla \widehat{u_{ht}}(z)\|=|z|^{1-\alpha}\|\nabla \hat{F}_{h}(z)A_hv_h\|\leq C |z|^{1-\alpha}|\,|z|^{\alpha/2-1}
\|A_h v_h\|\leq C|z|^{-\alpha/2} \|A_h v_h\|.
\end{equation}
On substitution of (\ref{Q1aa}) and (\ref{Q1}) in \eqref{xi-estimate-1}, we use (\ref{m013}) to obtain
\begin{eqnarray}\label{xi-estimate-2}
\|\xi(t)\|+h\|\nabla\xi(t)\|&\leq & Ch^2 \left(\int_{\Gamma} |e^{zt}|\, |z|^{-1}\,|dz|\right)\|A_h v_h\| \nonumber \\
&\leq & Ch^2  \left(\int_{1/t}^\infty e^{\rho t\cos\theta}\rho^{-1}d\rho
+\int_{-\theta}^{\theta}e^{\cos\psi}d\psi\right)\|A_h v_h\| \nonumber \\
&\leq & C h^2 \|A_hv_h\|.
\end{eqnarray}
Now,  by the identity ${A}_hR_h=P_hA$, we have
$$
\|A_h R_h v\| = \|P_hA v\| \leq \|A v\|=|v|_2,
$$
which shows the estimate \eqref{0e2} for $q=2.$

For the case $q=1$, that is, $v\in \dot{H}^1(\Omega),$
consider (\ref{Q1aa}) and 
 the identity 
 $$\widehat{u_{ht}}(z)=z\hat{u}_{h}(z)-v_h$$ 
 to obtain using (\ref{m3})
\begin{equation}\label{m0156}
\|\nabla \hat{u}_{ht}(z)\|=\|\nabla(z\hat{F}_{h}(z)v_h-v_h)\|\leq (M+1) \|\nabla v_h\|.
\end{equation}
From the estimate \eqref{xi-estimate-1}, using \eqref{Q1aa} and \eqref{m0156} 
with  $\|\nabla v_h\|=\|\nabla R_h v\|\leq\|\nabla v\|,$ we deduce that 
\begin{eqnarray*}
\|\xi(t)\|+h\|\nabla\xi(t)\|&\leq & Ch^2 \left(\int_{\Gamma} |e^{zt} |z|^{\alpha/2-1}\,|dz|\right)|v|_1  \\
&\leq & Ch^2  \left(\int_{1/t}^\infty e^{\rho t\cos\theta}\rho^{\alpha/2-1}d\rho+\int_{-\theta}^{\theta}e^{\cos\psi}t^{-\alpha/2}d\psi\right)|v|_1 \\
&\leq & Ct^{-\alpha/2}h^2 |v|_1.
\end{eqnarray*}
This completes the proof for the case $q=1.$

Since estimates for $q=1$ and $q=2$ are known, then  interpolation technique provides
result for $q\in [1,2].$ 
This concludes the rest of the proof. \end{proof}

\begin{remark}\label{r2}  
Note that the estimate  \eqref{e1}  in Theorem \ref{thm:H2} remains valid when $v_h= P_h v$. Indeed, for $q=2$,
let $\tilde{u}_h$ denote the solution of \eqref{semi-FV} with 
$v_h=P_h v.$  Then $\zeta:=\tilde{u}_h-\bar{u}_h$ satisfies 
\begin{equation*}\label{e7}
\zeta_t+\partial_t^{1-\alpha}\bar{A}_h\zeta=0,\quad  t>0,\quad  \zeta(0)=P_h v-R_h v.
\end{equation*}
Since 
$$\zeta(t)=-\frac{1}{2 \pi i} \int_{\Gamma} e^{zt}{\hat{E}}_h(z)(P_h v-R_h v) \; dz,$$
we deduce
$$
\|\zeta(t)\|\leq C\;\| P_h v-R_h v\| \int_{\Gamma} |e^{zt}|\; |z|^{-1} |d\;z|\leq C h^2 | v|_2.
$$
Thus, the estimate \eqref{e1} with $q=2$ follows  by the triangle inequality.  If the inverse 
inequality $\|\nabla \chi\|\leq Ch^{-1}\|\chi\|$ holds, which is the case if the mesh is quasi-uniform,  
then the estimate in the  gradient norm follows directly for  $v_h=P_h v$.
 
If the $L^2(\Omega)$-projection operator $P_h$ is stable in $\dot H^1(\Omega)$, i.e., 
$\|\nabla P_hw\|\leq C |w|_1$, then the estimate \eqref{e1} holds for the case 
$q=1$ and the choice   $v_h=P_hv$. A sufficient condition for such stability of 
$P_h$ is the quasi-uniformity of the mesh. 
Now, by interpolation the  estimate \eqref{e1} holds for $q\in [1,2]$  and $v_h=P_hv$.
\end{remark}


\subsection {Error estimates for nonsmooth initial data}\label{sec:L2}
In this subsection, we establish optimal error estimates for the semidiscrete FVE scheme (\ref{semi-FV}) for nonsmooth 
initial data $v\in L^2(\Omega)$.
\begin{theorem} \label{thm:L2}
Let $u$ and $\bar{u}_h$ be the solution of $(\ref{main})$ and $(\ref{semi-FV})$, respectively, with 
$v\in L^2(\Omega)$ and $v_h =P_hv$. Then, there exists a positive constant $C$, independent of $h,$ such that
\begin{equation}\label{m12a}
\|\bar{u}_h(t)-u(t)\| + \|\nabla(\bar{u}_h(t)-u(t))\|\leq Cht^{-\alpha}\|v\|, \quad t>0.
\end{equation}
Furthermore, if the quadrature error operator $Q_h$ satisfies \eqref{sym}, then the following optimal 
error estimate holds:
\begin{equation}\label{m12b}
\|\bar{u}_h(t)-u(t)\| \leq Ch^2t^{-\alpha}\|v\|, \quad t>0.
\end{equation}
\end{theorem}
\begin{proof} As before, it is sufficient to prove estimates for $\xi$.
We first apply  (\ref{0}) with $p=0$ to arrive at
 $$\|\hat{E}_h(z)\bar{A}_hQ_h\widehat{u_{ht}}\|\leq C|z|^{\alpha-1}\|Q_h\widehat{u_{ht}}\|.$$    
Then, the following bound follows from the integral representation  (\ref{m013}):
\begin{equation}\label{0m13}
\|\xi(t)\|\leq C\int_{\Gamma} |e^{zt}||z|^{\alpha-1}\|Q_h\widehat{u_{ht}}(z)\|\,|dz|.
\end{equation} 
To estimate the gradient of $\xi$,  we note that 
$$\|\nabla \hat{E}_h(z)\bar{A}_hQ_h\widehat{u_{ht}}\|\leq C|z|^{\alpha-1}\|\nabla Q_h\widehat{u_{ht}} \|,$$
and hence,
\begin{equation}\label{0m14}
\|\nabla\xi(t)\|\leq C\int_{\Gamma} |e^{zt}||z|^{\alpha-1}\|\nabla Q_h\widehat{u_{ht}}(z)\|\,|dz|.
\end{equation}
Note that $\|Q_h\widehat{u_{ht}}\|\leq Ch \|\widehat{u_{ht}} \|$  holds on a general mesh, and 
 $\|\nabla Q_h\widehat{u_{ht}}\|\leq Ch \|\widehat{u_{ht}} \|$ by \eqref{QQ}. 
Since $\|\widehat{u_{ht}}(z)\|=\|z\hat{F}_{h}(z)v_h-v_h\|\leq C \|v_h\|$ by (\ref{m3}), a substitution into 
\eqref{0m13}  and \eqref{0m14} yields the first estimate \eqref{m12a}. 
Finally, if (\ref{sym}) holds, then  \eqref{m12b}  follows immediately from \eqref{0m13}, 
which completes the proof.
\end{proof}



\subsection {$L^\infty(\Omega)$-error estimates}\label{sec:Linfty}
In the following, we obtain a superconvergence result for the gradient of $\xi$ in the $L^2(\Omega)$-norm.
As a consequence, assuming $v\in L^{\infty}(\Omega)$ and  the quasi-uniformity on the mesh, a quasi-optimal error estimate 
in the stronger  $L^\infty(\Omega)$-norm is derived for the semidiscrete FVE solution $\bar u_h$. 
We first prove the following Lemma by refining some of the  estimates derived in the proofs of 
Theorem \ref{thm:H2}. 
\begin{lemma} \label{thm:Linfty-1}
For $1\leq q\leq 2$, and  with  $v_h =R_hv$, there is a positive constant $C,$ independent of $h,$
such that 
\begin{equation*}\label{12a}
 \|\nabla \xi(t)\|\leq Ch^2t^{-\alpha(3-q)/2}|v|_q, \quad t>0.
\end{equation*}
The estimate is still valid for $v_h =P_hv,$ but with quasi-uniform assumption on the mesh.
\end{lemma}
\begin{proof} 
By using  bounds \eqref{0} and \eqref{QQ}, we   obtain instead of \eqref{g2} the following estimate
$$
\|\nabla \hat{E}_h(z)\bar{A}_hQ_h\widehat{u_{ht}}(z)\|\leq C |z|^{\alpha-1}\|\nabla Q_h\widehat{u_{ht}}(z)\|
\leq  C h^2 |z|^{\alpha-1}\|\nabla \widehat{u_{ht}}(z)\|.
$$
Since $\|\nabla\widehat{u_{ht}}(z)\|\leq c|z|^{-\alpha/2}\|A_hv_h\|$ by \eqref{Q1}, we note  from  
the representation  \eqref{m013} that 
\begin{equation*}\label{0gg2}
\|\nabla\xi(t)\| \leq Ch^2 |v|_2 \int_{\Gamma} |e^{zt} |z|^{\alpha/2-1}\,|dz|\leq  Ct^{-\alpha/2}h^2 |v|_2.
\end{equation*} 
Similarly, taking into account \eqref{m0156}, we obtain
\begin{equation*}\label{0gg2-b}
\|\nabla\xi(t)\| \leq Ch^2 |v|_1 \int_{\Gamma} |e^{zt} |z|^{\alpha-1}\,|dz|\leq  Ct^{-\alpha}h^2 |v|_1.
\end{equation*} 
Now, the desired estimate \eqref{12a} for $q\in [1,2]$ follows by  interpolation which completes  the proof.
\end{proof}

Note that for 2D-problems,  the Sobolev inequality 
$$ \|\chi\|_{L^{\infty}(\Omega)} \leq C\; |\ln h| \; \|\nabla\chi\| \;\;\;\forall \chi\in V_h,$$
and  Lemma 4.2  imply for $q\in [1,2]$ that
\begin{equation}
\label{max-norm}
\|\xi(t)\|_{L^{\infty}(\Omega)} \leq C\,|\ln h| \, \|\nabla\xi(t)\| \leq 
C|\ln h| \,h^2t^{-\alpha(3-q)/2}|v|_q.
\end{equation}

As a consequence, we obtain the following quasi-optimal $L^\infty(\Omega)$-error estimate  
by combining the results in \eqref{max-norm} and \eqref{FE-3}.
\begin{theorem}\label{thm:Linfty}
Let $u$ and $\bar{u}_h$ be the solution of $(\ref{main})$ and $(\ref{semi-FV})$, respectively, with 
$v_h =P_hv$.  Assume that
 $v \in \dot H^q(\Omega)\cap L^\infty(\Omega)$ for  $1\le  q\le 2$.  Then, under the quasi-uniformity condition on the mesh, there holds
\begin{equation*}\label{FV-3}
 \|\bar u_h(t) - u(t)\|_{L^{\infty}(\Omega)} \leq
 C|\ln h|^{\frac{5}{2}} h^2 t^{-\alpha(3-q)/2}\Big(|v|_q+\|v\|_{L^\infty(\Omega)}\Big),\qquad 1\le q\le 2.
\end{equation*}
\end{theorem}

\section {Fully discrete schemes}\label{sec:discrete}
\se
In this section, we analyze two fully discrete schemes for the semidiscrete problem \eqref{semi-FV} 
using the framework of  convolution quadrature developed in \cite{LST-1996,Lubich-2006}, which has
been initiated in \cite{Lubich-1986,Lubich-1988}. To describe this framework, we first divide the 
time interval $[0,T]$ into $N$ equal subintervals with a time step size $\tau=T/N$, and let $t_j=j\tau$. 
Then, the convolution quadrature \cite{Lubich-1986} refers to an approximation of any function of 
the form $k\ast\varphi$ as
$$
(k\ast\varphi) (t_n):=\int_0^{t_n}k(t_n-s)\varphi(s)\,ds\approx \sum_{j=0}^n \beta_{n-j}(\tau) \varphi(t_j),
$$
where the convolution weights $\beta_j=\beta_j(\tau)$ are computed from the Laplace transform $\hat{k}(z)$  of 
$k$ rather than the kernel $k(t)$. 
This method provides, in particular, an interesting tool for  approximating the Riemann-Liouville fractional integral  of order $\alpha$, $\partial_t^{-\alpha}\varphi := \omega_\alpha\ast \varphi$,  where 
$\omega_\alpha(t)=t^{\alpha-1}/\Gamma(\alpha)$. Here,  $\hat{k}(z)=\hat{\omega}_\alpha(z)=z^{-\alpha}$.

With $\partial_t$ being time differentiation, we define $\hat{k}(\partial_t)$ as the operator 
of (distributional) convolution with the kernel $k$: $\hat{k}(\partial_t)\varphi=k\ast \varphi$ for a 
function $\varphi(t)$ with suitable smoothness.
A convolution quadrature approximates $\hat{k}(\partial_t)\varphi$ by a discrete convolution 
$\hat{k}(\bar \partial_\tau)\varphi$ at $t=t_n$  as
$$
\hat{k}(\bar \partial_\tau)\varphi(t_n) = \sum_{j=0}^n \beta_{n-j}(\tau) \varphi(t_j),
$$
where the quadrature weights $\{\beta_j(\tau)\}_{j=0}^{\infty}$ are determined by the generating 
power series
$$
\sum_{j=0}^\infty  \beta_j(\tau) \xi^j=\hat{k}(\delta(\xi)/\tau)
$$
with $\delta(\xi)$ being a rational function, chosen as the quotient of the generating polynomials of 
a stable and consistent linear multistep method. In this paper, we consider the Backward Euler (BE) 
and the second-order backward difference (SBD) methods, for which $\delta(\xi)=1-\xi$ and $\delta(\xi)=(1-\xi)+ (1-\xi)^2/2$, respectively.
For the BE method,  the convolution quadrature formula for approximating the fractional integral $\partial_t^{-\alpha}\varphi$ is given by 
$$
\bar \partial_\tau^{-\alpha}\varphi(t_n) = \sum_{j=0}^n  \beta_{n-j} \varphi(t_j), \text{ where } 
\sum_{j=0}^\infty  \beta_{j} \xi^j=[(1-\xi)/\tau]^{-\alpha}, \quad \beta_j=\tau^\alpha(-1)^{j}
\left(\begin{array}{c}
-\alpha\\
j
\end{array}\right),
$$
while for the SBD method, the quadrature weights are provided by the formula \cite{Lubich-1986}:
$$
\beta_j=\tau^\alpha(-1)^{j}\left(\frac{2}{3}\right)^\alpha\sum_{l=0}^j
3^{-l}
\left(\begin{array}{c}-\alpha\\j-l\end{array}\right)
\left(\begin{array}{c}-\alpha\\l\end{array}\right).
$$
An important property of the convolution quadrature is that it maintains some relations of the continuous convolution. 
For instance, the associativity of convolution is valid for the convolution 
quadrature  \cite{Lubich-2004} such as
\begin{equation}\label{s1}
\hat{k}_1(\bar \partial_\tau) \hat{k}_2(\bar \partial_\tau)=\hat{k}_1 \hat{k}_2(\bar \partial_\tau)\quad 
\text{ and }\quad 
\hat{k}_1(\bar \partial_\tau)(k\ast\varphi)=(\hat{k}_1(\bar \partial_\tau)k)\ast\varphi.
\end{equation}


In the following lemma, we state an interesting result on the error of the convolution quadrature, see \cite[Theorem 4.1]{Lubich-1988} and \cite[Theorem 2.2]{Lubich-2004}.
\begin{lemma}\label{lem:Lubich}
Let $G(z)$ be  analytic in the sector $\Sigma_\theta$ and such that
\begin{equation*}\label{s3}
\|G(z)\|\leq M|z|^{-\mu}\quad \forall z\in \Sigma_\theta,
\end{equation*}
for some real $\mu$ and $M$. Assume that the linear multistep method is strongly $A$-stable and of 
order $p\geq 1$. Then, for $\varphi(t)=ct^{\nu-1}$, the convolution quadrature satisfies
\begin{equation}\label{s4}
\|G(\partial_t)\varphi(t) - G(\bar \partial_\tau)\varphi(t)\|  \leq \left\{
\begin{array}{ll}
C t^{\mu-1+\nu-p} \tau^p, & \nu\geq p\\
C t^{\mu-1} \tau^\nu, & 0< \nu\leq p.
\end{array} \right.
\end{equation}
\end{lemma}


\subsection{Error analysis for the BE method}
In this subsection,  we specify the construction of a fully discrete scheme based on the BE method for the 
semidiscrete problem \eqref{semi-FV}.  Then, we derive $L^2(\Omega)$-error estimates for smooth 
and nonsmooth initial data. 

After integrating in time from $0$ to $t$, the semidiscrete scheme \eqref{FVP} takes the form
\begin{equation}\label{s5}
\bar{u}_h+\partial_t^{-\alpha} \bar A_h\bar{u}_h=v_h.
\end{equation}
The second term on the left-hand side is a convolution, and then, it  can be approximated 
at $t_n=n\tau$ with $U_h^n$ by
\begin{equation}\label{s6}
U_h^n+\bar\partial_\tau^{-\alpha} \bar A_h U_h^n=v_h.
\end{equation}
The symbol $\bar\partial_\tau^{-\alpha}$ refers to the relevant convolution quadrature generated 
by the BE method. 

Thus, with $U_h^0=v_h$, the fully discrete solution can be represented by
\begin{equation}\label{BE}
U_h^n=\left(I+\beta_0\bar A_h\right)^{-1} \left(U_h^0-\sum_{j=0}^{n-1}\beta_{n-j}\bar A_h U^j\right) \quad \mbox{ for } n\geq 1.
\end{equation}
We notice that the term corresponding to $j=0$ in the formula can be omitted without affecting the convergence rate of the scheme \cite{LST-1996}.


In view of \eqref{s5} and \eqref{s6}, we can write the  error  $U^n_h-\bar{u}_h(t_n)$ at $t=t_n$ as 
\begin{equation*}\label{s7}
U^n_h-\bar{u}_h(t_n)= \left( G(\bar\partial_\tau)- G(\partial_t)\right)v_h,
\end{equation*}
where $G(z)=(I+z^{-\alpha}\bar A_h)^{-1}$. Using the identity 
$$(I+z^{-\alpha}\bar A_h)^{-1} = I-(z^{\alpha}I+\bar A_h)^{-1}\bar A_h,$$
and denoting $\bar G(z)=-(z^{\alpha}I+\bar A_h)^{-1}$, the error  can be represented as
\begin{equation}\label{s8}
U^n_h-\bar{u}_h(t_n)= \left( \bar G(\bar\partial_\tau)- \bar G(\partial_t)\right) \bar A_hv_h.
\end{equation}
Using  Lemma \ref{lem:Lubich}, we now derive the following error estimates.
%
\begin{lemma}\label{lem:BE}
Let $\bar{u}_h$ and $U^n_h$ be the solutions of problems $(\ref{semi-FV})$ and $(\ref{s6})$, respectively, with 
$U^0_h=v_h$. Then, the following estimates hold:

(a) If $v\in \dot{H}^2(\Omega)$ and $v_h =R_hv$, then
\begin{equation}\label{s10}
\|U^n_h-\bar{u}_h(t_n)\|\leq C \tau t_n^{\alpha-1}|v|_2.
\end{equation}

(b) If $v\in L^2(\Omega)$ and $v_h =P_hv$, then
\begin{equation}\label{s9}
\|U^n_h-\bar{u}_h(t_n)\|\leq C \tau t_n^{-1}\|v\|.
\end{equation}

\end{lemma}
\begin{proof} For the  estimate \eqref{s10}, we recall that, by \eqref{res1},  
$\|\bar G(z)\|\leq M_\theta |z|^{-\alpha}\;  \forall z\in \Sigma_\theta.$ An application of Lemma \ref{lem:Lubich}  (with $\mu=\alpha$, $\nu=1$ and $p=1$)   to \eqref{s8}  yields
$$
\|U^n_h-\bar{u}_h(t_n)\|\leq C \tau t_n^{\alpha-1}\|\bar A_hv_h\|.
$$
Now, we introduce a projection operator 
$\bar P_h:L^2(\Omega)\rightarrow V_h$ defined by
$$
(\bar P_hw,\chi)_h=(w,\chi) \quad \forall \chi\in  V_h.
$$
Then,  $\bar P_h$ is  stable in $L^2(\Omega)$ and  the identity $\bar A_hR_h=\bar P_hA$ holds, since
$$
(\bar A_h R_h w,\chi)_h=(\nabla R_h w,\nabla \chi)=(\nabla w, \nabla\chi)=(A w,\chi)=(\bar P_h A w,\chi)_h \quad \forall \chi \in V_h.
$$
As $v_h=R_h v$, it follows that 
$$
\|\bar A_h v_h\| =\|\bar A_h R_h v\| = \|\bar P_h A v\| \leq C \|A v\|=C |v|_2,
$$
which shows \eqref{s10}.

For the estimate \eqref{s9}, we notice that
$ \|{G}(z)\|=|z|^\alpha\|(z^{\alpha}I+\bar A_h)^{-1}\|\leq M_\theta\;  \forall z\in \Sigma_\theta.$
Then, by applying Lemma \ref{lem:Lubich} (with $\mu=0$, $\nu=1$ and $p=1$) to \eqref{s7}, we obtain 
$$
\|U^n_h-\bar{u}_h(t_n)\|\leq C \tau t_n^{-1}\|v_h\|.
$$
Now, the estimate  follows from  the $L^2(\Omega)$-stability of $P_h$. 
This completes the rest of the  proof.
\end{proof}
\begin{remark}\label{rem:BE} For $v\in \dot{H}^2(\Omega)$, we can choose  $v_h =P_hv$. Let $\tilde U_h^n$ be the  solution of the fully discrete scheme \eqref{s6} with $v_h =P_hv$. Then, by the stability of the scheme, a direct consequence of Lemma \ref{lem:BE}, we have
$\|U^n_h-\tilde U^n_h\|\leq \|R_hv-P_hv\|\leq  C h^2|v|_2,$ showing that
\begin{equation}\label{s11}
\|U^n_h-\bar{u}_h(t_n)\|\leq C (\tau t_n^{\alpha-1}+h^2)|v|_2.
\end{equation}
Hence, by interpolating \eqref{s9} and \eqref{s11} it follows that for  $v_h =P_hv$,
\begin{equation}\label{s11a}
\|U^n_h-\bar{u}_h(t_n)\|\leq C (\tau t_n^{-1})^{1/2} (\tau t_n^{\alpha-1}+h^2)^{1/2}|v|_1.
\end{equation}
\end{remark} 

As a consequence of Lemma \ref{lem:BE}, we obtain error estimates for the fully discrete scheme \eqref{BE} with smooth and nonsmooth initial data. 

\begin{theorem}\label{thm:BE}
Let $u$ and $U^n_h$ be the solutions of problems $(\ref{main})$ and $(\ref{s6})$, respectively, with 
$U^0_h=v_h$. Then, the following error estimates hold:

(a) If $v\in \dot{H}^2(\Omega)$ and $v_h =R_hv$, then
\begin{equation}\label{s12}
\|U^n_h-u(t_n)\|\leq C (h^2+\tau t_n^{\alpha-1})|v|_2.
\end{equation}

(b) If $v\in \dot{H}^1(\Omega)$, $v_h =P_hv$ and the mesh is quasi-uniform, then
\begin{equation}\label{s13}
\|U^n_h-u(t_n)\|\leq C (h^2t_n^{-\alpha/2}+\tau t_n^{-1+\alpha/2})|v|_1.
\end{equation}

(c) If $v\in L^2(\Omega)$, $v_h =P_hv$ and $Q_h$ satisfies \eqref{sym}, then
\begin{equation}\label{s14}
\|U^n_h-u(t_n)\|\leq  C (h^2t_n^{-\alpha}+\tau t_n^{-1})\|v\|.
\end{equation}
\end{theorem}
\begin{proof} The first estimate \eqref{s12} follows from \eqref{e1}, \eqref{s10} and the triangle inequality, while 
the third estimate \eqref{s14}  follows from \eqref{m12b}  and \eqref{s9}. By combining \eqref{e1} (with $q=1$) which holds for $v_h =P_hv$ and \eqref{s11a}, we deduce
\begin{equation*}\label{s15}
\|U^n_h-u(t_n)\|\leq C (h^2t_n^{-\alpha/2}+\tau t_n^{-1+\alpha/2}+\tau^{1/2}t_n^{-1/2}h)|v|_1.
\end{equation*}
An inspection of the three terms between brackets shows that the  square of the third term  equals 
the product of the first two terms, which proves the estimate  \eqref{s13}. This concludes the proof.
\end{proof}


\subsection{Error analysis for the SBD method}
Now we consider the time discretization of  $(\ref{semi-FV})$ constructed with the convolution 
quadrature based on the second-order backward difference formula. 
From Lemma \ref{lem:Lubich}, it is obvious that one can get only a first-order error bound if,  
for instance, $\varphi$ is constant (i.e., $\nu=1$). In order to overcome this difficulty, a correction 
of the scheme is needed. Below,  we present  modifications of the convolution quadrature  based on 
the strategy in \cite{LST-1996} and \cite{Lubich-2006}.
By noting the identity
\begin{equation*}\label{k5}
 (I+\partial_t^{-\alpha} \bar A_h)^{-1}=I-(I+\partial_t^{-\alpha} \bar A_h)^{-1}\partial_t^{-\alpha} \bar A_h,
\end{equation*}
it turns out from \eqref{s5} that the semidiscrete solution $\bar{u}_h$ can be rewritten as
\begin{equation*}\label{k5}
\bar{u}_h=v_h-(I+\partial_t^{-\alpha} \bar A_h)^{-1}\partial_t^{-\alpha} \bar A_h v_h.
\end{equation*}
This leads to the modified convolution quadrature \cite{Lubich-2006}
\begin{equation}\label{k6}
U_h^n= v_h-(I+\bar\partial_\tau^{-\alpha} \bar A_h)^{-1}\partial_t^{-\alpha} \bar A_h v_h,
\end{equation}
where the exact contribution $\partial_t^{-\alpha} A_h v_h = \omega_{\alpha+1}(t)A_h v_h$ 
is kept in the new formula (\ref{k6})  in order to improve the time accuracy. The symbol $\bar\partial_\tau^{-\alpha}$ refers to the  convolution quadrature generated by the SBD method. Unfortunately, this correction would not yield  optimal time accuracy. 
A second choice for the  modified convolution quadrature which will be considered here is based on  
the approximation \cite{LST-1996}
\begin{equation}\label{k6n}
U_h^n= v_h-(I+\bar\partial_\tau^{-\alpha} \bar A_h)^{-1}\bar \partial_\tau^{1-\alpha} \partial_t^{-1}\bar A_h v_h,
\end{equation}
where  the term $\partial_t^{-1}$ is kept to achieve second-order time accuracy. The advantages of both  numerical methods  \eqref{k6} and \eqref{k6n} are described in \cite{Lubich-2006}. 

For the  numerical implementation, it is essential to write \eqref{k6n} as a time stepping algorithm. 
Let $1_\tau=(0,3/2,1,\cdots)$ so that $1_\tau=\bar\partial_\tau\partial_t^{-1}1$ at grid point $t_n$.
Then by applying the operator $(I+\bar\partial_\tau^{-\alpha} \bar A_h)$ to both sides of \eqref{k6n} and using the associativity of convolution in \eqref{s1}, we arrive at  the equivalent form  
\begin{equation*}\label{k6n2}
(I+\bar\partial_\tau^{-\alpha}\bar A_h)(U_h^n-v_h)=-\bar\partial_\tau^{-\alpha} \bar A_h 1_\tau v_h.
\end{equation*}
By applying again the operator $\bar\partial_\tau$, we obtain 
\begin{equation}\label{SBD}
\bar\partial_\tau(U_h^n-v_h)+\bar\partial_\tau^{1-\alpha} \bar A_h(U_h^n-v_h)=-\bar\partial_\tau^{1-\alpha}\bar A_h 1_\tau v_h.
\end{equation}
By noting that $1v_h- 1_\tau v_h=(v_h,-1/2v_h,0,\cdots)$, we thus define the time stepping scheme as: with $U^0_h=v_h$, find $U_h^n$ such that 
$$
\frac{3}{2}\tau^{-1}(U_h^1-U_h^0)+\tilde\partial_\tau^{1-\alpha} \bar A_hU_h^1=0,
$$
and for $n\geq 2$
$$
\bar\partial_\tau U_h^n+\tilde\partial_\tau^{1-\alpha} \bar A_hU_h^n=0,
$$
where the modified convolution quadrature $\tilde\partial_\tau^{1-\alpha}$ is given by \cite{LST-1996}
$$
\tilde\partial_\tau^{1-\alpha}\varphi^n=
\left( \sum_{j=1}^n \beta_{n-j}^{(1-\alpha)}\varphi^j+ \frac{1}{2}\beta_{n-1}^{(1-\alpha)}\varphi^0\right),
$$
with the weights $\{\beta_j^{(1-\alpha)}\}$ being generated by the SBD method.

Now using Lemma \ref{lem:Lubich}, we derive the following error bounds for smooth and nonsmooth initial data.
\begin{lemma}\label{lem:SBD}
Let $\bar{u}_h$ and $U^n_h$ be the solutions of problems $(\ref{semi-FV})$ and $(\ref{SBD})$, respectively, and set 
$U^0_h=v_h$. Then, the following estimates hold:

(a) If $v\in \dot{H}^2(\Omega)$ and $v_h =R_hv$, then
\begin{equation}\label{k10}
\|U^n_h-\bar{u}_h(t_n)\|\leq C \tau^2 t_n^{\alpha-2}|v|_2.
\end{equation}

(b) If $v\in L^2(\Omega)$ and $v_h =P_hv$, then
\begin{equation}\label{k9}
\|U^n_h-\bar{u}_h(t_n)\|\leq C \tau^2 t_n^{-2}\|v\|.
\end{equation}
\end{lemma}
\begin{proof} For the  estimate \eqref{k10}, we set  
$$\bar G(z)=z^{1-\alpha}(I+z^{-\alpha}\bar A_h)^{-1}$$
and write the error as
\begin{equation}\label{k9b}
U^n_h-\bar{u}_h(t_n)= \left( \bar G(\bar\partial_\tau)- \bar G(\partial_t)\right)\partial_t^{-1} \bar A_h v_h.
\end{equation}
Since $\|\bar G(z)\|\leq M_\theta |z|^{1-\alpha}\;  \forall z\in \Sigma_\theta$  by \eqref{res1},  \eqref{k9b} 
and  Lemma \ref{lem:Lubich}  (with $\mu=\alpha-1$, $\nu=2$ and $p=2$) imply 
$$
\|U^n_h-\bar{u}_h(t_n)\|\leq c\tau^2 t_n^{\alpha-2}\|\bar A_hv_h\|.
$$
Then, the desired estimate \eqref{k10} follows from the identity $\bar A_hR_h=\bar P_hA$.

For the estimate \eqref{k9}, we note  
with 
$$\bar{G}(z)=z^{1-\alpha}(I+z^{-\alpha}\bar{A}_h)^{-1}\bar A_h$$ 
and using \eqref{k6n} that
\begin{equation}\label{k9a}
U^n_h-\bar{u}_h(t_n)= \left( \bar {G}(\bar\partial_\tau)- \bar{G}(\partial_t)\right)\partial_t^{-1} v_h.
\end{equation}
Since $\|\bar{G}(z)\|\leq M_\theta |z|\;  \forall z\in \Sigma_\theta$, a use of 
\eqref{k9a}, Lemma \ref{lem:Lubich} (with $\mu=-1$, $\nu=2$ and $p=2$) and the $L^2(\Omega)$ stability 
of $P_h$ yield the estimate \eqref{k9}.
This completes the rest of the proof
\end{proof}
\begin{remark} By the stability of the scheme, a direct consequence of Lemma \ref{lem:SBD}, and the arguments 
in Remark \ref{rem:BE}, the following error estimate holds for  $v_h= P_h v$
\begin{equation}\label{k11}
\|U^n_h-\bar{u}_h(t_n)\|\leq  C (\tau^2 t_n^{\alpha-2}+h^2)|v|_2.
\end{equation}
Then,  by interpolation of  \eqref{k9} and \eqref{k11} we get for  $v_h =P_hv$ 
\begin{equation*}\label{k11a}
\|U^n_h-\bar{u}_h(t_n)\|\leq C (\tau^2 t_n^{-2})^{1/2} (\tau t_n^{\alpha-2}+h^2)^{1/2}|v|_1.
\end{equation*}
\end{remark} 

Using the estimates derived in Sections \ref{sec:H2} and \ref{sec:L2} for the semidiscrete problem, and  following the arguments  in the proof of Theorem \ref{thm:BE}, we can now state  the error estimates for the fully discrete scheme \eqref{SBD} with smooth and nonsmooth initial data. 

\begin{theorem}\label{thm:SBD}
Let $u$ and $U^n_h$ be the solutions of problems $(\ref{main})$ and $(\ref{SBD})$, respectively, with 
$U^0_h=v_h$. Then, the following error estimates hold:

(a) If $v\in \dot{H}^2(\Omega)$ and $v_h =R_hv$, then
\begin{equation*}\label{k12}
\|U^n_h-u(t_n)\|\leq C (h^2+\tau^2 t_n^{\alpha-2})|v|_2.
\end{equation*}

(b) If $v\in \dot{H}^1(\Omega)$, $v_h =P_hv$ and the mesh is quasi-uniform, then
\begin{equation*}\label{k13}
\|U^n_h-u(t_n)\|\leq  C (h^2t_n^{-\alpha/2}+\tau^2 t_n^{\alpha/2-2})|v|_1.
\end{equation*}

(c) If $v\in L^2(\Omega)$, $v_h =P_hv$ and $Q_h$ satisfies \eqref{sym}, then
\begin{equation*}\label{k14}
\|U^n_h-u(t_n)\|\leq C (h^2t_n^{-\alpha}+\tau^2 t_n^{-2})\|v\|.
\end{equation*}\end{theorem}


\section {On extensions}\label{sec:extensions}
\se
In this section, we discuss the extension of our analysis
to other type of problems including those with more general linear elliptic operator and
other time-fractional evolution problems. We only concentrate on  the error analysis of the  
semidiscrete FVE method. Completely discrete schemes can be discussed in a similar way
 by choosing appropriate convolution quadratures and following the analysis in Section \ref{sec:discrete}. 

\subsection{Problems with more general elliptic operators}

More precisely, we consider problem 
\eqref{semi-FE} with 
$$A u =-\nabla\cdot(\kappa(x) \nabla u) + c(x) u,$$
where $\kappa(x)$ is a symmetric, positive definite  $2\times 2$ matrix function on $\bar\Omega$ with smooth entries and 
$c(x)\in L^\infty(\Omega)$ and $c(x)\geq c_0 >0.$
The corresponding  bilinear form $a(\cdot,\cdot):H_0^1(\Omega)\times H_0^1(\Omega)\rightarrow \mathbb{R}$ 
becomes
\begin{equation*}\label{bilinear-n}
a (w,\chi) = (\kappa(x) \nabla w,\nabla \chi)+(c(x)w, \chi)\;\;\;\forall \chi\in H^1_0(\Omega).
\end{equation*}
The natural generalization of the finite volume element  method \eqref{FV}  yields 
\begin{equation*} \label{Ah-n}
a_h(w,\chi)=  \sum_{P_i\in N_h^0} \chi(P_i)\left(- \int_{\partial K_{P_i}^*}
(\kappa\nabla w)\cdot{\bf n}\,ds +\int_{ K_{P_i}^*}c(x)w\, dx\,ds\right)
\quad \forall w \in V_h,\, \chi\in  \vhs.
\end{equation*}
In general, the bilinear form $a_h(w,\Pi_h^\ast\chi)$, $\chi \in V_h$,  is  not symmetric on $V_h$. 
However, if $\kappa$ and $c$ are constant over each element of the triangulation $ \mathcal{T}_h,$ 
then the bilinear form takes the form, see \cite{BR-1987},
$$a_h(w,\Pi_h^\ast\chi) = (\kappa(x) \nabla w,\nabla \chi)+(c(x)w, \Pi_h^\ast\chi)\quad \forall w,\chi\in V_h,$$
which is  symmetric since $(c(x)w, \Pi_h^\ast\chi)=(c(x)\chi, \Pi_h^\ast w)$. As symmetry is important 
in our analysis, we shall consider the modified bilinear form, see \cite{CLT-2013},
\begin{equation*} \label{Ah-2}
\tilde a_h(w,\chi)=  \sum_{P_i\in N_h^0} \chi(P_i)\left(- \int_{\partial K_{P_i}^*}
(\tilde\kappa(x)\nabla w)\cdot{\bf n}\,ds +\int_{ K_{P_i}^*}\tilde c(x) w\, dx\,ds\right)
\quad \forall w \in V_h,\, \chi\in  \vhs,
\end{equation*}
where, for each $x\in K$, $K\in \mathcal{T}_h$, $\tilde\kappa(x)=\kappa(x_K)$ and $\tilde c(x)=c(x_K)$, 
with $x_K$ being the barycenter of the element $K$. Now, the FVE method reads: 
find $\tilde u_h(t)\in V_h$ such that 
\begin{equation} \label{semi-FV-n}
(\tilde{u}_h',\chi)_h+ \tilde a_h(\Ba \tilde{u}_h,\Pi_h^\ast\chi)=  0\quad
\forall \chi\in V_h,\quad t\in (0,T], \quad \tilde{u}_h(0)=v_h.
\end{equation}
Introducing the  discrete operator $\tilde A_h:V_h\rightarrow V_h$ by
\begin{equation} \label{Dh-n}
(\tilde A_h w,\chi)_h=\tilde a_h(w,\Pi_h^\ast\chi)  \quad \forall w,\chi\in V_h,
\end{equation}
we rewrite \eqref{semi-FV-n}  as
\begin{equation} \label{FVP-n}
\tilde u_h'(t)+\Ba \tilde A_h \tilde u_h(t)=  0, \quad t>0, \quad \tilde u_h(0)=v_h.
\end{equation}
 
Following our analysis in Section \ref{sec:error}, with
$\xi(t)=\tilde{u}_h(t)-u_h(t)$, we split the error
 $\tilde{u}_h(t)-u(t)= (u_h(t)-u(t))+\xi(t)$, where it is well known that $u_h(t)-u(t)$ and $\nabla(u_h(t)-u(t))$ are  estimated by the analogues of \eqref{FE-1}-\eqref{FE-2}.
 It is, therefore, sufficient to 
 derive estimates for $\xi$, which satisfies for $t\geq 0$ 
 \begin{equation} \label{semi-FV-nn}
(\xi',\chi)_h+ \tilde a(\Ba \xi,\Pi_h^\ast\chi)=  -\epsilon_h(u_{ht},\chi)-\tilde\epsilon_h(u_h,\chi)\quad
\forall \chi\in V_h,\quad \tilde{u}_h(0)=v_h,
\end{equation}
where $\epsilon_h(\cdot,\cdot)$ is defined in  \eqref{m8}  and  $\tilde \epsilon_h(\cdot,\cdot)$ is given by
\begin{equation}\label{m8-n}
\tilde\epsilon_h(w,\chi)= \tilde a_h(w,\Pi_h^\ast\chi)-a(w,\chi) \quad \forall w,\chi \in V_h.
\end{equation}
Upon introducing the quadrature error operators  $Q_h:V_h\rightarrow V_h$ and $\tilde Q_h:V_h\rightarrow V_h$ defined by
\begin{equation}\label{m8-nn}
\tilde a_h(Q_h w,\Pi_h^\ast\chi)=\epsilon_h(\chi,\psi) \quad \text{and} \quad 
\tilde a_h(\tilde Q_h w,\Pi_h^\ast\chi)=\tilde\epsilon_h(\chi,\psi)\quad \forall w,\chi \in V_h,
\end{equation}
the equation  \eqref{semi-FV-nn} can be rewritten in the operator form  as
\begin{equation}\label{m10-n}
\xi_t(t)+\partial_t^{1-\alpha}\tilde{A}_h\xi(t)=-\tilde{A}_hQ_hu_{ht}(t)- \tilde{A}_h\tilde Q_hu_{h}(t),\quad  t>0,\quad  \xi(0)=0.
\end{equation}
To derive estimates for $\xi$, we need the following bound, see  \cite {CLT-2013} for a proof.
\begin{lemma}\label{lem:Qh-n}
Let $\tilde A_h$, $Q_h$ and $\tilde Q_h$ be the operators defined in \eqref{Dh-n} and \eqref{m8-nn}. Then
\begin{equation}\label{QQ-n}
\|\nabla Q_h\chi\|+h\|\tilde{A}_hQ_h\chi\|\leq Ch^{p+1}\|\nabla^{p}\chi\|\quad \forall \chi\in V_h, \quad p=0,1,
\end{equation}
and similar  result holds for the operator $\tilde Q_h$.
\end{lemma}
%


 Now, we show the following estimates. 
 \begin{theorem} \label{thm:H2-1} For the error $\xi$ defined by $(\ref{m10-n})$, there is a positive constant $C,$
 independent of $h,$ such that for $t>0$,  
 \begin{equation}\label{e1-m}
\|\xi(t)\|+h\|\nabla\xi(t))\|\leq C\max\{t^{1-\alpha/2},t^{1-\alpha}\}h^2\|A_hv_h\|,
\end{equation}
\begin{equation}\label{e1-mm}
\|\xi(t)\|+h\|\nabla\xi(t))\|\leq Ct^{1-\alpha/2}h^2\|\nabla v_h\|,
\end{equation}
and
\begin{equation}\label{e1-mmm}
\|\xi(t)\|+h\|\nabla\xi(t))\|\leq Ct^{1-\alpha}h\|v_h\|.
\end{equation}
If $\tilde Q_h$ satisfies $\|\tilde Q_h\chi\|\leq Ch^2\|\chi\|$ $\forall \chi \in V_h$, then
\begin{equation}\label{e1-4m}
\|\xi(t)\|\leq Ct^{1-\alpha}h^2\|v_h\|.
\end{equation}
\end{theorem}
\begin{proof}
By taking Laplace transforms in \eqref{m10-n},  we represent $\xi(t)$  by
\begin{equation}\label{m10-nn}
\begin{aligned}
 \xi(t)=-\frac{1}{2\pi i}\int_{\Gamma} e^{zt} & \hat{E}_h(z)\tilde{A}_hQ_h\hat{u}_{ht}(z)\,dz \\
&-\frac{1}{2\pi i}\int_{\Gamma} e^{zt}\hat{E}_h(z)\tilde{A}_h\tilde Q_h\hat{u}_{h}(z)\,dz =: \xi_1+\xi_2,
\end{aligned}
\end{equation}
where $\hat{E}_h(z)=z^{\alpha-1}(z^\alpha I+\tilde{A}_h)^{-1}$. The first term $\xi_1$ is bounded as 
in the proofs of Theorems \ref{thm:H2} and \ref{thm:L2} using Lemma \ref{lem:Qh-n} instead of 
Lemma \ref{lem:Qh}. To bound the second term $\xi_2$, 
we notice that, similar to \eqref{Q1aa}, we arrive at
\begin{equation}\label{Q1aa-m}
\|\hat{E}_h(z)\tilde{A}_h\tilde Q_h\hat u_h(z)\|+h\|\nabla \hat{E}_h(z)\tilde{A}_h\tilde Q_h\hat u_h(z)\|\leq Ch^2 
|z|^{\alpha/2-1}\|\nabla \hat{u}_h(z)\|.
\end{equation}
Using the identity 
$$
\hat E_h(z)=z^{-1}[I-\hat E_h(z)\tilde A_h]
$$
and \eqref{11}, it follows that
\begin{eqnarray}\label{Q2-a}
\|\nabla\hat E_h(z)v_h\|&\leq& |z|^{-1}[\|\nabla v_h\|+\|\nabla\hat  E_h(z)\tilde A_hv_h\|]\nonumber\\
&\leq & C |z|^{-1}[\|\tilde A_hv_h\|+|z|^{\alpha/2-1}\|\tilde A_hv_h\|].
\end{eqnarray}
Substituting  \eqref{Q2-a} in \eqref{Q1aa-m} and using the integral representation of $\xi_2$ in \eqref{m10-nn}, 
we obtain the estimate \eqref{e1-m}. To derive \eqref{e1-mm},  a  use of \eqref{m3} yields
$$
\|\nabla\hat E_h(z)v_h\|\leq C |z|^{-1} \|\nabla v_h\|.
$$
Then, the bound follows immediately. For the last cases \eqref{e1-mmm} and \eqref{e1-4m}, we apply \eqref{00} to get 
$$\|\hat{E}_h(z)\bar{A}_h\tilde Q_h\hat{u}_h\|_p\leq C|z|^{\alpha-1}\|\tilde Q_h\hat{u}_h\|_p,\quad p=0,1.$$ 
Then, the left-hand side in \eqref{Q1aa-m} is bounded by
\begin{equation*}\label{Q3a-n}
 C|z|^{\alpha-1}( \|\tilde Q_h\hat{u}_h(z)\| + h \|\nabla \tilde Q_h\hat{u}_h(z)\|).
\end{equation*}
Using Lemma \ref{lem:Qh-n} and the fact that $\|\hat{u}_h(z)\|\leq |z|^{-1}\|v_h\|$, we obtain the desired 
results by following the arguments  in the proof of Theorem \ref{thm:L2}. This completes the proof of the theorem.
\end{proof}

\subsection{Other time-fractional evolution problems}
Our analysis can be applied to obtain optimal FVE error estimates for  other type of time-fractional 
evolution problems. 
This may include, for instance, evolution equations with memory terms of convolution type:
\begin{equation} \label{EE}
u'(x,t)+\mathcal{I}^\alpha A u(x,t)=0, \quad \alpha\in(0,1),
\end{equation}
see \cite{LST-1996}, which is also called fractional diffusion-wave equation, 
the following parabolic integro-differential equation with singular kernel of the type
\begin{equation}\label{PIDE-S}
u'(x,t)+ (I + \mathcal{I}^\alpha) A u(x,t)=0, \quad \alpha\in(0,1),
\end{equation}
see, \cite{MST2006},
and the Rayleigh-Stokes problem described by the time-fractional differential equation
\begin{equation} \label{RS}
u'(x,t)+(I+\gamma\partial_t^\alpha) A u(x,t)=0,\quad \alpha\in(0,1),
\end{equation}
which has been considered in \cite{EJLZ2016}. Here $\gamma$ is a positive constant. 
In order to unify  problems \eqref{EE}-\eqref{RS}, we define $\mathcal{J}^\alpha$ denoting 
a time integral/differenial operator and consider the unified problem by 
\begin{equation} \label{J}
u'(x,t)+\mathcal{J}^\alpha A u(x,t)=0.
\end{equation}
Now an application of Laplace transforms  in \eqref{J} yields
$$
z\hat{u}+h(z)A\hat{u}=v,
$$
with some function $h(z)$ depending on $\alpha$. Hence, we formally have, $\hat{u}= (z+ h(z) A)^{-1} v =: \hat{E}_h(z)v.$

Let $\bar{A}_h$ and $Q_h$ be the  operators defined in Section \ref{sec: FVM}.  
Then, the FVE method reads: find $\bar u_h(t)\in V_h$ such that
\begin{equation} \label{semi-GG1}
\bar{u}_h'+  \mathcal{J}^\alpha \bar A_h\bar{u}_h=  0\quad
\quad t\in (0,T], \quad \bar{u}_h(0)=v_h.
\end{equation}
Again using the corresponding  FE solution $u_h,$ we split $\bar{u}_h- u := (u_h-u) + (\bar{u}_h- u_h)
=:(u_h-u) + \xi,$ where $\xi$ satisfies
the similar representation formula 
\begin{equation}\label{m013-nnn}
\xi(t)=-\frac{1}{2\pi i}\int_{\bar{\Gamma}_\theta} e^{zt}\hat{E}_h(z)\bar{A}_hQ_h\hat{u}_{ht}(z)\,dz.
\end{equation} 
Note that in this case the operator $\hat{E}_h(z)$ is  given by
\begin{equation}\label{sm1}
\hat{E}_h(z)=\beta(z)(z\beta(z)I+\bar A_h)^{-1},
\end{equation} 
and $\beta(z)=h(z)^{-1}$. For the problem \eqref{EE}, we observe that $\beta(z)=z^{\alpha},$  for the problem
\eqref{PIDE-S}, $\beta(z)= z^\alpha/(1+z^{\alpha}),$ and for 
the problem \eqref{RS}, $\beta(z)=1/(1+\gamma z^\alpha)$.  We assume that one can properly choose $\theta$ in $(\pi/2,\pi)$ such that $z\beta(z)\in \Sigma_{\theta'}$ for all $z\in \Sigma_{\theta}$ where the angle $\theta'\in(\pi/2,\pi)$. This is indeed
 possible in all given examples. With this, the resolvent estimate yields
\begin{equation}\label{sm2}
\|(z\beta(z)I+\bar A_h)^{-1}\|\leq \frac{M_{\theta'}}{|z\beta(z)|}\quad \forall z\in \Sigma_{\theta},
\end{equation}
where $M_{\theta'}=1/\sin(\pi-\theta')$.   Therefore, from \eqref{sm1},
\begin{equation}\label{sm3}
\|\hat{E}_h(z)\|\leq M_{\theta'}|z|^{-1} \quad \forall z\in \Sigma_{\theta}.
\end{equation} 
Following  arguments from \cite{LST-1996},  we deduce that
\begin{equation}\label{sm2}
\|\bar A_h\hat{E}_h(z)\|\leq C_{\theta'}|\beta(z)|\,\quad  \forall z\in \Sigma_{\theta}.
\end{equation} 
Now, we can prove the  analogous of Lemma \ref{lem:Ah}.
\begin{lemma}\label{lem:Ah2-n} Let $\hat E_h(z)$ be given by \eqref{sm1}. 
With $\chi\in V_h$, the following estimates hold:
\begin{equation}\label{0-n}
 \|{\bar A_h\hat{E}_h(z)\chi}\|\leq C_{\theta'} |\beta(z)|^{1-p/2}|z|^{-p/2} \;\|\bar {A}_h^{p/2}{\chi}\| 
 \quad \forall z\in \Sigma_\theta,\quad 0\leq p\leq 2,
 \end{equation} 
 \begin{equation}\label{1-n}
|{\hat{E}_h(z)\chi}|_1\leq C_{\theta'} |\beta(z)|^{1/2}|z|^{-1/2} \|{\chi}\|\quad \forall z\in \Sigma_\theta,
 \end{equation}  
 where $C_{\theta'}$ is independent of the mesh size $h$.
\end{lemma}
\begin{proof}
We obtain the first estimate \eqref{0-n} by interpolating \eqref{sm3} 
and  \eqref{sm2}. The second estimate follows from the fact that
$$
\|\nabla (z\beta(z)I+\bar A_h)^{-1}\chi\|\leq C |z\beta(z)|^{-1/2} \|\chi\|\quad \forall \chi \in V_h,
$$
 see (2.13) in \cite{Lubich-2006}.
\end{proof}

In the following theorem,  optimal error estimates  
are obtained for  smooth and nonsmooth initial data $v\in \dot{H}^q(\Omega)$, $q=0,1,2.$
\begin{theorem} \label{thm:H2-n} 
%
For the error  $\xi$ defined by  $\eqref{m013-nnn}$, there is a positive  constant $C$, independent of
$h,$ such that $t>0$, 
\begin{equation}\label{e1-n}
\|\xi(t)\|+h\|\nabla\xi(t))\|\leq Ch^2\|\bar A_hv_h\|.
\end{equation}
If $|\beta(z)|\leq C |z|^\mu$ $\forall z\in\Sigma_\theta$ for some real $\mu<1$, then
\begin{equation}\label{e1-nn}
\|\xi(t)\|+h\|\nabla\xi(t))\|\leq Ct^{-(\mu+1)/2}h^2\|\nabla v_h\|.
\end{equation}
If $|\beta(z)|\leq C |z|^\mu$ $\forall z\in\Sigma_\theta$ and $\bar Q$  satisfies \eqref{sym}, then
\begin{equation}\label{e1-nnn}
\|\xi(t)\|+h\|\nabla\xi(t))\|\leq Ct^{-(\mu+1)}h^2\| v_h\|.
\end{equation}

\end{theorem}
\begin{proof} We will only prove the estimate in the $L^2(\Omega)$-norm. The estimate in the gradient norm is derived in a similar way.
We shall make use of the estimate  \eqref{xi-estimate-1} obtained in the proof of  Theorem \ref{thm:H2}.

When $q=2$, that is, $v\in \dot{H}^2(\Omega),$ apply (\ref{0-n}) with $p=1$ and (\ref{1-n}) in Lemma \ref{lem:Ah2-n} to get
\begin{equation*}\label{m014-n}
\|\hat{E}_h(z)\bar{A}_hQ_h\widehat{u_{ht}}(z)\|\leq 
   C  |\beta(z)|^{1/2} |z|^{-1/2}\|\nabla Q_h \widehat{u_{ht}}(z)\|,
\end{equation*}
and
\begin{equation*}\label{g2-n}
\|\nabla \hat{E}_h(z)\bar{A}_hQ_h\widehat{u_{ht}}(z)\|\leq C  |\beta(z)|^{1/2} |z|^{-1/2}\|\bar A_hQ_h\widehat{u_{ht}}(z)\|. 
\end{equation*}
Then, by (\ref{QQ}) in Lemma \ref{lem:Qh}, we deduce 
\begin{equation}\label{Q1aa-n}
\|\hat{E}_h(z)\bar{A}_hQ_h\widehat{u_{ht}}(z)\|+h\|\nabla \hat{E}_h(z)\bar{A}_hQ_h\widehat{u_{ht}}(z)\|\leq 
Ch^2 |\beta(z)|^{1/2} |z|^{-1/2}\|\nabla \widehat{u_{ht}}(z)\|.
\end{equation}
Since 
$$\widehat{u_{ht}}(z)=-h(z)\bar A_h\hat{{u}}_{h}(z)= -h(z)\bar A_h\hat{F}_{h}(z)v_h,$$ 
an  estimate analogous to (\ref{1-n}) yields
\begin{eqnarray}
\|\nabla \widehat{u_{ht}}(z)\|&= &|h(z)|\|\nabla \hat{F}_{h}(z)\bar A_hv_h\|\nonumber\\
&\leq & C|h(z)|\,|\beta(z)|^{1/2} |z|^{-1/2}\|\bar A_hv_h\|\nonumber\\
&\leq &  C |\beta(z)|^{-1/2} |z|^{-1/2} \|\bar  A_h v_h\|.\nonumber
\end{eqnarray}
Thus, the left-hand side in (\ref{Q1aa-n}) is  bounded by $|z|^{-1} \|\bar A_h v_h\|$. Now, substitution in
\eqref{xi-estimate-1} gives the desired estimate.


For $q=1$, we notice that in view of \eqref{sm3},  the bound (\ref{m0156}) holds, and therefore 
substitution in  (\ref{Q1aa-n}) gives the  new upper bound $Ch^2 |z|^{\mu/2-1/2} \|\nabla v_h\|$ in (\ref{Q1aa-n}). 
The estimate \eqref{e1-nn} follows then by integration.

Finally, for $q=0$, we have by \eqref{sm2},
$$\|\hat{E}_h(z)\bar{A}_hQ_h\widehat{u_{ht}}\|\leq C|\beta(z)|\,\|Q_h\widehat{u_{ht}}\|
\leq C|z|^{\mu}\|Q_h\widehat{u_{ht}}\|.$$ 
In view of \eqref{sm3}, we have $\|\widehat{u_{ht}}(z)\|=\|z\hat{F}_{h}(z)v_h-v_h\|\leq C \|v_h\|$. Therefore, 
if \eqref{sym} is satisfied then $\|\hat{E}_h(z)\bar{A}_hQ_h\widehat{u_{ht}}\|
\leq C h^2|z|^{\mu}\|\widehat{u_{ht}}\|\leq C h^2|z|^{\mu}\|v_h\|.$
  Now, \eqref{e1-nnn} follows by integration and this concludes the rest of the proof.
\end{proof}

By interpolating \eqref{e1-n} and \eqref{e1-nnn} we obtain for $q\in [0,2]$  
$$
\|\xi(t)\|+h\|\nabla\xi(t))\|\leq Ct^{-(\mu+1)(1-q/2)}h^2\|\bar A_h^{q/2} v_h\|,\quad  t>0.
$$
Notice that $\mu=\alpha$ for problems \eqref{EE} and \eqref{PIDE-S},  while $\mu=-\alpha$ for the  Rayleigh-Stokes problem \eqref{RS}. Hence, for the Rayleigh-Stokes problem the previous estimate reads:
$$
\|\xi(t)\|+h\|\nabla\xi(t))\|\leq Ct^{-(1-\alpha)(1-q/2)}h^2\|\bar A_h^{q/2} v_h\|,\quad  t>0,
$$
provided \eqref{sym} is satisfied.



We finally  consider the  following class of time-fractional order  diffusion problems:
\begin{equation}\label{eq-caputo}
^C\partial_t^{\alpha}u (x,t)+\cL u(x,t)=0,
\end{equation}
where  $^C\partial_t^{\alpha}$ is the  fractional Caputo derivative of order $\alpha\in(0,1)$.
For this class of equations, optimal error estimates for the semidiscrete  FE method  have been 
established in \cite{JLZ2013}. The FVE method applied to \eqref{eq-caputo} is to seek  $\bar{u}_h\in V_h$ such that 
\begin{equation*} \label{semi-GG}
\Bac\bar{u}_h + \bar A_h\bar{u}_h=  0\quad
\quad t\in (0,T], \quad \bar{u}_h(0)=v_h.
\end{equation*}
Again a  comparison between  the FE solution and FVE solution along with Laplace techniques and 
semigroup type properties as has been done in Section \ref{sec:error} yields {\it a priori} FVE error estimates for the 
fractional order evolution problem \eqref{eq-caputo} for both smooth and nonsmooth initial data. Since the
proof technique is similar to the tool used in Section $4$, we skip the details.


%

\begin{figure}[h]
\begin{center}
  \caption{Triangular meshes with $M=8$, (a) symmetric mesh (b) nonsymmetric mesh.}  
  \label{Fig:meshes}
   \includegraphics[width=6cm, height=6cm]{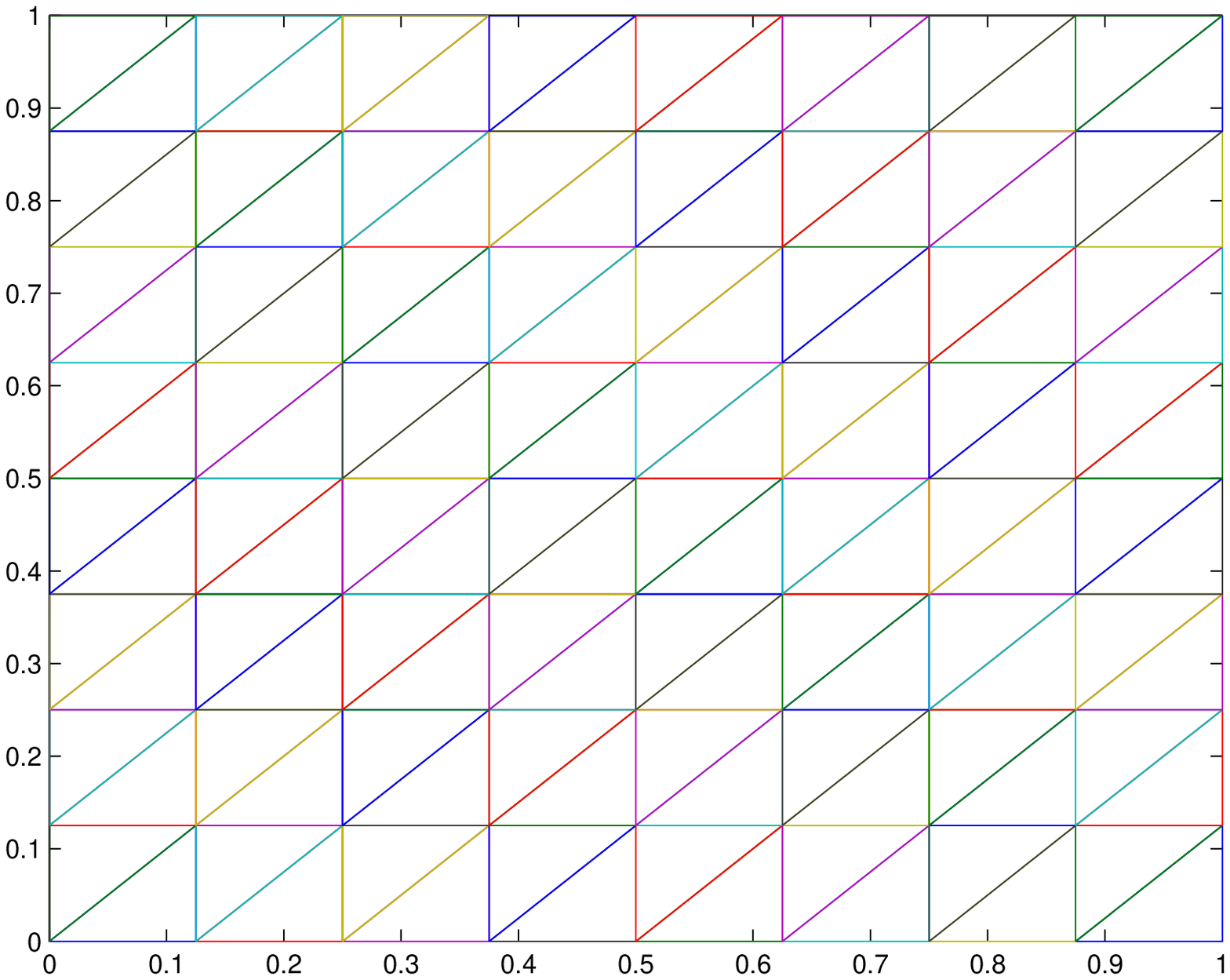}
   $\;$
  \includegraphics[width=6cm, height=6cm]{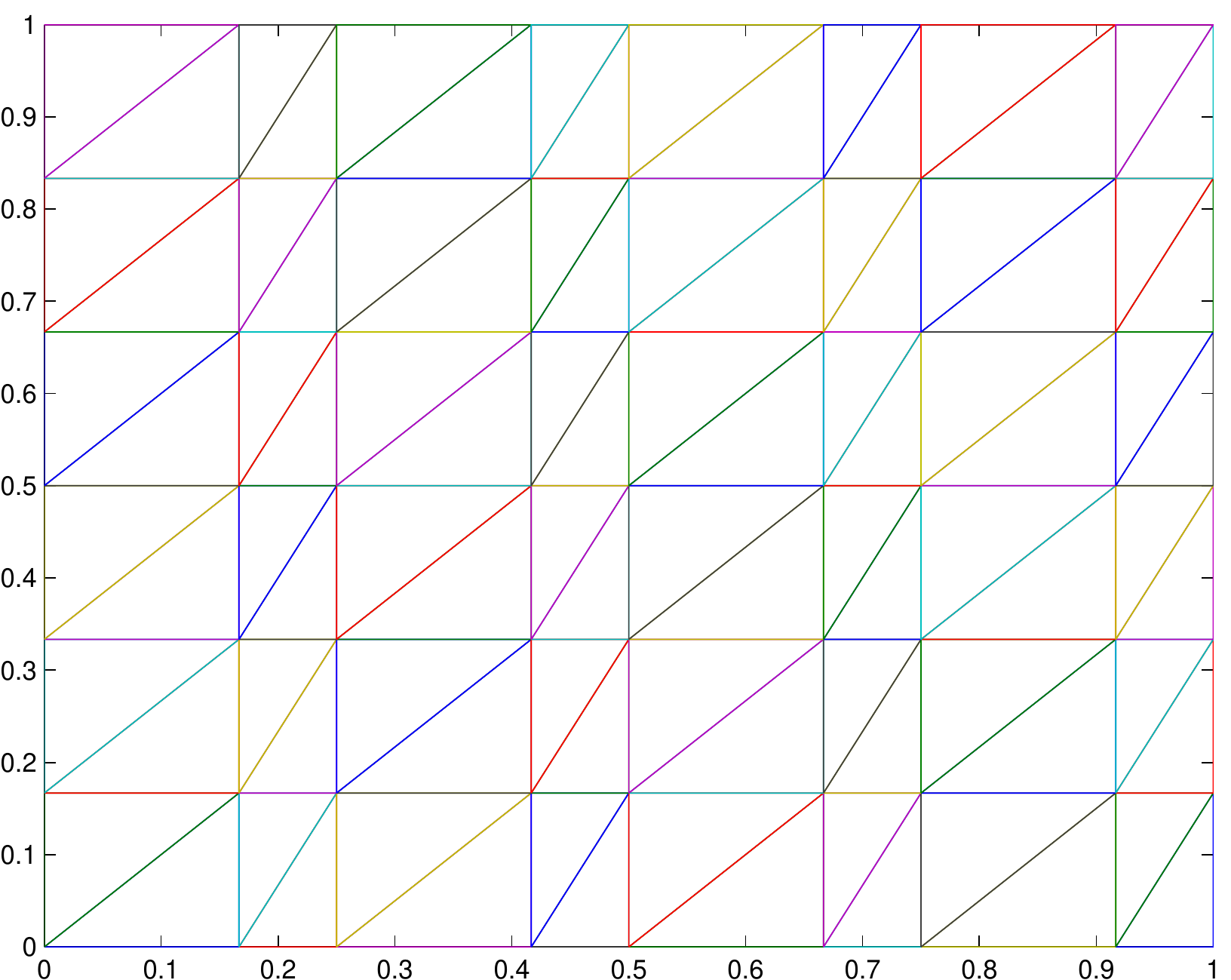}
\end{center}
\end{figure}

\subsection{Derivation by the lumped mass FE method} 
In this subsection, we extend our analysis to  the lumped mass FE method applied to the 
time-fractional diffusion problem \eqref{main}.
For completeness,  we briefly describe, below, this approximation.
For $K\in\mathcal{T}_h$ with vertices $P_i$, $i=1,2,3$, consider the quadrature formula
\begin{equation*}\label{QF}
Q_{K,h}(f)=\frac{|K|}{3}\sum_{i=1}^3f(P_i)\approx\int_Kf\,dx.
\end{equation*}
Then, we define an approximation of the $L^2$-inner product on $V_h$ by
\begin{equation*}\label{QF}
\langle w,\chi\rangle=\sum_{K\in\mathcal{T}_h} Q_{K,h}(w\chi).
\end{equation*}
The lumped mass Galerkin FE method reads: find $\bar u_h(t)\in V_h$ satisfying
\begin{equation*} \label{semi-LM}
\langle\bar{u}_h',\chi\rangle+  a(\Ba \bar{u}_h,\chi)=  0\quad
\forall \chi\in V_h,\quad t\in (0,T], \quad \bar{u}_h(0)=v_h.
\end{equation*}
In operator form, the method  can be written as 
\begin{equation*} \label{LM-OF}
 \bar u_h'(t)+\Ba \bar A_h  \bar u_h(t)=  0, \quad t>0, \quad  u_h(0)=v_h,
\end{equation*}
where $\bar A_h:V_h\rightarrow V_h$ is the discrete Laplacian corresponding to the inner product 
$\langle\cdot,\cdot\rangle$   given by
\begin{equation} \label{LM-OF-2}
\langle\bar A_h w,\chi\rangle=(\nabla w,\nabla \chi)  \quad \forall w,\chi\in V_h.
\end{equation}
%
Now, introduce $\xi(t)=\tilde{u}_h(t)-u_h(t)$ with $u_h(t)$ being the Galerkin FE solution. Then    
$\xi$  satisfies 
 \begin{equation*} \label{LM-OF-1}
 \xi'(t)+\Ba \bar A_h  \xi(t)=  -\bar A_hQ_hu_{ht}, \quad t>0, \quad  \xi(0)=0, 
\end{equation*}
where $Q_h:V_h\rightarrow V_h$ is the quadrature error defined by
\begin{equation}\label{m8-LM}
(\nabla Q_h\chi,\nabla\psi)=\epsilon_h(\chi,\psi):=\langle\chi,\psi\rangle-(\chi,\psi)\quad \forall \psi \in V_h.
\end{equation}
Since the operators $\bar{A}_h$ and $Q_h$ defined by \eqref{LM-OF-2} and \eqref{m8-LM} have  properties 
similar  to the corresponding operators  in the FVE method  in Section \ref{sec:error}, (see also \cite{CLT-2012}), 
then the error estimates  for the lumped mass FE method and  their proofs are quite analogous  to the results proved
in Sections \ref{sec:error} and \ref{sec:discrete} for the FVE method. Therefore, we can easily derive   optimal error estimates and
 we shall not pursue it further. 



\section{ Numerical Experiments}  \label{sec:NE}
\se


In this section, we present some numerical tests to validate our theoretical results. We choose  $\Omega=(0,1)\times (0,1)$ and perform the
computation on two families of symmetric and nonsymmetric triangular meshes. The  symmetric meshes are uniform with  mesh  size $h=\sqrt{2}/M$, where $M$ is the number of equally spaced subintervals in both the $x$- and $y$-directions, see Figure \ref{Fig:meshes}(a). For the nonsymmetric  meshes,  we choose $M$ subintervals in the $x$-direction
and $3M/4$  equally spaced subintervals  in the $y$-direction with the assumption that $M$ is divisible by 4. The intervals in the $x$-direction are  of  lengths $4/3M$ and $2/3M$ and   distributed such that they form an alternating series as shown in Figure \ref{Fig:meshes}(b). One can notice that the  nonsymmetric mesh  defines a triangulation that is not symmetric at any vertex, see  \cite[Section 5]{CLT-2013} for more details.


We consider three numerical examples with smooth and nonsmooth initial data. By separation of variables, the exact solution of problem \eqref{main} can represented by a rapidly converging Fourier series
\begin{equation}\label{eq: u series}
u(x,y,t)=2\sum_{m,n=1}^\infty (v, \phi_{mn})
    E_{\alpha}(-\lambda_{mn} t^{\alpha})\phi_{mn}(x,y),
\end{equation}
 where 
 $E_{\alpha}(t):=\sum_{p=0}^\infty\frac{t^p}{\Gamma(\alpha p+1)}$ is the
 Mittag-Leffler function and 
$$
\phi_{mn}(x,y)=2\sin(m \pi x)\sin(n \pi y)
\quad\text{and}\quad
\lambda_{mn}=(m^2+n^2)\pi^2\quad{\rm for}~~ m\,,n=1, 2, \ldots
$$
are the orthonormal eigenfunctions and
corresponding eigenvalues of $-\Delta$ subject to homogeneous Dirichlet boundary conditions.
In our computation, we evaluate the exact solution by truncating the Fourier series in \eqref{eq: u series} 
after $60$ terms.

\begin{table}[t]
\begin{center}
\caption{$L^2$-error for cases (a)-(c) on symmetric meshes, $\alpha=0.75$, $h=1/400$.}
\label{table:1}
\begin{tabular}{|r|cccc|}
\hline
$N$&
BE & rate& SBD & rate
\\
\hline
\multicolumn{5}{|c|}{Case (a)}\\
\hline
 5& 4.8880e-003 &         &  1.3161e-003 & \\
 10& 2.1844e-003 &   1.16  &  3.1530e-004 & 2.06\\
 20& 1.0367e-003 &   1.08 &  7.2627e-005 & 2.12\\
 40& 5.0547e-004 &   1.04 &  1.6922e-005 & 2.10\\
 80& 2.4952e-004 &  1.02 &  3.6949e-006 & 2.18\\
\hline
\multicolumn{5}{|c|}{Case (b)} \\
\hline
  5   &4.8270e-003 &          &1.3857e-003 &  \\       
  10  &2.1578e-003 &1.16  &3.3341e-004 &2.06\\
  20  &1.0247e-003 &1.07  &7.7019e-005 &2.11\\
  40  &5.0021e-004 &1.03  &1.7736e-005 &2.19\\
  80  &2.4751e-004 &1.02  &3.6842e-006 &2.27\\

\hline
\multicolumn{5}{|c|}{Case (c)}\\
\hline
 5  & 2.9708e-003  &          & 8.2449e-004 &       \\   
 10 &  1.3300e-003 & 1.16  & 2.0483e-004 &  2.01\\
 20 &  6.3206e-004 & 1.07  & 4.7324e-005 &  2.11\\
 40 & 3.0862e-004  & 1.03  & 1.0961e-005 &  2.11\\
 80 & 1.5275e-004  & 1.01  & 2.4291e-006 &  2.17\\
\hline
\end{tabular}
\end{center}
\end{table}

We consider the following initial data to illustrate the convergence theory.
\begin{itemize}
\item[ (a)] With  $v=xy(1-x)(1-y)$, its Fourier sine coefficients become
\[
(v,\phi_{mn})=8(1-(-1)^m)(1-(-1)^n)(mn\pi^2)^{-3},\quad {\rm for}~~m,n=1,\,2,\ldots.
\] 
This example represents the smooth case as $v \in \dot H^{2}(\Omega)$.
\item[ (b)] For this example, choose $v=xy\chi_{(0,1/2]\times(0,1/2]} +(1-x)y\chi_{(1/2,1)\times(0,1/2]}+
 x(1-y)\chi_{(0,1/2]\times(1/2,1)}+(1-x)(1-y)\chi_{(1/2,1)\times(1/2,1)}$,  where $\chi_D$ denotes 
 the characteristic function on the domain $D$. This initial data  is less smooth compared to 
 the previous case. One can verify that its Fourier coefficients are given by 
\[
(v,\phi_{mn})=2(1-(-1)^m)(1-(-1)^n)(mn\pi^2)^{-2}(-1)^{mn},\quad {\rm for}~~m,n=1,\,2,\ldots.
\]
Note that $v \in \dot H^{1+\epsilon}(\Omega)$ for $0\le \epsilon<1/2$.
\item[(c)] With  $v=\chi_{(0,1/2[\times(0,1)}(x,y)$, its Fourier sine coefficients become
\[
(v,\phi_{mn})=2(1-\cos(m\pi/2))(1-(-1)^n)(mn\pi^2)^{-1},\quad {\rm for}~~m,n=1,\,2,\ldots.
\]
Here,  $v \in \dot H^{\epsilon}(\Omega)$ for $0\le \epsilon<1/2$. 
\end{itemize}

\begin{table}
\begin{center}
\caption{Errors for cases (a)-(c) on symmetric meshes, $\alpha=0.75$, $\tau=1/500$.}
\label{table:2}
\begin{tabular}{|r|cccc|}
\hline
$M$&
$L^2$-norm error & rate&$L^\infty$-norm error& rate\\
\hline
\multicolumn{5}{|c|}{Case (a)}\\
\hline
  8 & 1.4556e-003 &   &  1.0596e-004 &  \\
  16 & 3.7356e-004 &  1.96 &  2.7366e-005 &  1.95\\
  32 & 9.3259e-005 &  2.00 &  6.8602e-006 &  2.00\\
  64 & 2.2546e-005 &  2.05 &  1.6792e-006 &  2.03\\
  128 & 4.8155e-006 &  2.23 &  3.8055e-007 &  2.14\\
\hline
\multicolumn{5}{|c|}{Case (b)}\\
\hline
  8 &  8.9301e-004 &     &  2.0405e-004 &  \\
  16 &  2.2952e-004 &  1.96 &  5.5397e-005 &  1.88\\
  32 &  5.7285e-005 &  2.00 &  1.4340e-005 &  1.95\\
  64 &  1.3820e-005 &  2.05 &  3.5649e-006 &  2.01\\
  128 &  2.9842e-006 &  2.21 &  8.0446e-007 &  2.15\\
\hline
\multicolumn{5}{|c|}{Case (c)}\\
\hline
  8   &  7.1870e-004 &        &  2.7011e-003 &  \\
  16 &  1.8148e-004 &  1.99 &  8.7438e-004 &  1.63\\
  32  &  4.5181e-005 &  2.01 &  2.7169e-004 &  1.69\\
  64  &  1.1033e-005 &  2.03 &  7.6187e-005 &  1.83\\
  128 &  2.6557e-006 &  2.05 &  2.0470e-005 &  1.90\\
\hline
\end{tabular}
\end{center}
\end{table}

To examine  the temporal accuracy of the proposed schemes, we employ a uniform temporal mesh with a time step $\tau=T/N$, where $T=0.5$ is the time of interest in all numerical experiments. 
We fix the mesh size $h$ at $h = 1/400$ so that the error incurred by spatial discretization
is negligible, which enable us to examine the temporal convergence rate.
The computation is performed on symmetric meshes. 
We measure the error $e^n =: u(t_n)-U^n$ by the normalized $L^2(\Omega)$-norm 
$\| e^n\|_{L^2(\Omega)}/\|v\|_{L^2(\Omega)}$. 
The numerical results are presented in  Table \ref{table:1} for the three proposed cases (a)-(c).
In the table, BE and SBD denote the convolution quadrature generated by the backward Euler  and 
the second-order backward difference methods, respectively. The {\tt rate} refers to the empirical convergence
rate, when the time step size $\tau$ halves. From the Table \ref{table:1}, a convergence rate of order 
$O(\tau)$ and $O(\tau^2)$ 
is observed for the BE  and  SBD schemes, respectively, and clearly both schemes exhibit a very steady 
behavior for both smooth and nonsmooth data, which agree well with our convergence
theory. Additional  numerical experiments with different values of fractional order $\alpha$ have shown 
similar  convergence rates. It was, in particular, observed that the error decreases as the fractional order $\alpha$ increases. 
More details on the behaviour of errors from BE and SBD methods combined with a Galerkin 
FE  discretization in space can be found in \cite{JLZ2016}.

To check the spatial  discretization error,  we fix the time step $\tau=1/500$ and use the SBD scheme so that the temporal discretization error is negligible. We carry out the computation on symmetric meshes.  
In Table \ref{table:2}, we list the normalized $L^2(\Omega)$-norm and $L^\infty(\Omega)$-norms of the error  for the      
 cases (a)-(c). The numerical results show a convergence rate $O(h^2)$ for the $L^2(\Omega)$-norm of the error for smooth and nonsmmoth initial data. A similar convergence rate is obtained in the $L^\infty(\Omega)$-norm (ignoring a logarithmic factor).
The results fully confirm the predicted rates on symmetric meshes. They also show the validity of the  convergence rate 
 in Theorem \ref{thm:Linfty} for case (c) where  $0<q<1$.



\begin{table}
\begin{center}
\caption{ Errors for case (c) on nonsymmetric meshes, $\alpha=0.75$, $\tau=1/500$.}
\label{table:Nonsym-1}
\begin{tabular}{|r|cccc|}
\hline
 & \multicolumn{4}{|c|}{FVEM}\\
\hline
$M$& $L^2$-norm error & rate&$L^\infty$-norm error& rate\\
\hline
   8  & 1.1209e-003  &   & 4.1704e-003 &   \\
  1.6& 2.7755e-004  & 2.01  & 1.3697e-003 &  1.61\\
  32 & 6.8036e-005  & 2.03  & 4.1953e-004 &  1.71\\
  64 & 1.6529e-005  & 2.04  & 1.1120e-004 &  1.92\\
  128& 3.9610e-006  & 2.06  & 3.0306e-005 &  1.88\\
  \hline
& \multicolumn{4}{|c|}{Lumped mass FEM}\\ 
 \hline
  8  &  1.1627e-003 &       &  4.1512e-003 &   \\
  16 &  3.1215e-004 &  1.90 &  1.3697e-003 &  1.60\\
  32 &  8.2238e-005 &  1.92 &  4.1472e-004 &  1.72\\
  64 &  2.1382e-005 &  1.94 &  1.1120e-004 &  1.90\\
  128&  5.8007e-006 &  1.88 &  3.3495e-005 &  1.73\\
  \hline
\end{tabular}
\end{center}
\end{table}

For nonsymmetric meshes, we are especially interested in spatial errors for nonsmooth initial data as the 
convergence theory suggests. In Table \ref{table:Nonsym-1}, we display the $L^2(\Omega)$- and 
$L^\infty(\Omega)$-norms of the error  for case (c) using  the FVE and the lumped mass FE   discretizations
 on  nonsymmetric meshes. The numerical results reveal that both discretizations exhibit a convergence rate 
 of order $O(h^2)$,  which may be seen as an unexpected result. However, as the initial data 
 $v\in \dot H^{1/2-\epsilon}(\Omega)$ for any $\epsilon>0$,  $v$ has some smoothness, and hence,  
 the numerical results do not contradict our theoretical findings. 
In addition, we notice that as the convergence rate is $O(h^2)$ for initial data  
in $\dot H^{1}(\Omega)$, by interpolation in $[0,1]$,  a convergence rate of order $O(h^{3/2})$ is 
expected for $v\in \dot H^{1/2}(\Omega)$. In our case, the smoothness of the particular initial data $v$ 
could then have  a positive effect on the convergence rate.

In \cite{CLT-2013}, the authors considered the nonsymmetric partition shown in Figure \ref{Fig:meshes}(b) and provided an initial  data for which the  optimal $L^2$-convergence does not hold. They proved that the best possible error bound in this case is of order 1, see Proposition 5.1 of \cite{CLT-2013}. 
Earlier in \cite{CLT-2012}, the same authors have established a one-dimensional example for which the 
$O(h^2)$ nonsmooth data error does not hold for the lumped mass FE  method. 
We, then, carried out our computation based on the example  in \cite[Proposition 5.1]{CLT-2013}. 
The numerical results are presented in Table \ref{table:Nonsym-2} using the SBD scheme. The error 
reported in the table represents the quantity $\xi(t)$ which measures the difference between the 
Galerkin FE solution and the FVE solution for the first set of numerical results and between the 
Galerkin FE solution and the lumped mass FE solution for the second set. As the nonsmooth data error 
from the standard Galerkin FE is always $O(h^2)$, the error from the considered methods is 
dominated by $\xi(t)$. From the Table \ref{table:Nonsym-2}, an order $O(h)$ of convergence rate is 
observed for both methods, which agrees well with the results in \cite{CLT-2013} and confirms our 
theoretical analysis.





\begin{table} 
\begin{center}
\caption{ Errors for case (d)  on nonsymmetric meshes, $\alpha=0.75$, $\tau=1/500$.}
\label{table:Nonsym-2}
\begin{tabular}{|r|cccc|}
\hline
 & \multicolumn{4}{|c|}{FVEM}\\
\hline
$M$& $L^2$-norm error & rate&$L^\infty$-norm error & rate \\
\hline
  8   & 9.9247e-005 &      & 3.8454e-004 & \\
  16  & 2.3133e-005 & 2.10 & 1.1238e-004 & 1.77\\
  32  & 1.1497e-005 & 1.01 & 4.8469e-005 & 1.21\\
  64  & 5.1181e-006 & 1.17 & 2.0545e-005 & 1.24\\
  128 & 2.5579e-006 & 1.00 & 9.7156e-006 & 1.08\\
 \hline
& \multicolumn{4}{|c|}{Lumped mass FEM}\\ 
 \hline
  8   & 4.4924e-004 &   & 1.7395e-003 &  \\
  16  & 1.0429e-004 &  2.11 & 5.0652e-004 &  1.78\\
  32  & 5.1762e-005 &  1.01 & 2.1821e-004 &  1.21\\
  64  & 2.3035e-005 &  1.17 & 9.2466e-005 &  1.24\\
  128 & 1.1511e-005 &  1.00 & 4.3722e-005 &  1.08\\
  \hline
\end{tabular}
\end{center}
\end{table}

For completeness, we extend our numerical study to examine some of the  problems presented in Section 6, namely; 
the subdiffusion problem \eqref{eq-caputo} with a fractional Caputo derivative  and the wave-diffusion problem  \eqref{EE}. 
The numerical solution in each case  is obtained by using the FVE method in space and a convolution quadrature in time 
generated by the second-order backward difference method. 
We run both examples  with the initial data $v$ given in case (c).

\begin{table}
\begin{center}
\caption{Numerical results for problem \eqref{eq-caputo}, $\alpha=0.75$, $\tau=1/500$.}
\label{table:5}
\begin{tabular}{|r|cccc|}
\hline
$M$&
$L^2$-norm error & rate&$L^\infty$-norm error& rate\\
\hline
  8   &  7.1870e-004 &        &  2.7011e-003 &  \\
  16 &  1.8148e-004 &  1.99 &  8.7438e-004 &  1.63\\
  32  &  4.5181e-005 &  2.01 &  2.7169e-004 &  1.69\\
  64  &  1.1033e-005 &  2.03 &  7.6187e-005 &  1.83\\
  128 &  2.6557e-006 &  2.05 &  2.0470e-005 &  1.90\\
\hline
\end{tabular}
\end{center}
\end{table}

For the first problem, we employ the second-order time discretization scheme derived in \cite[formula (2.16)]{JLZ2016}. 
The computed errors  are
presented in Table 5  and are clearly identical to the results in Table 2.
Even though it is known that the two representations \eqref{eq-caputo} and \eqref{a1}  are equivalent,  
the numerical methods obtained for each representation are in general different. 
However, in the current case, the fact that the time discrete schemes are equivalent is due to the
feature of the convolution quadrature, in particular, to the properties given in \eqref{s1}. 

\begin{figure}
    \centering
    \subfigure[$\alpha = 0.1$]
    {
        \includegraphics[width=4.2cm]{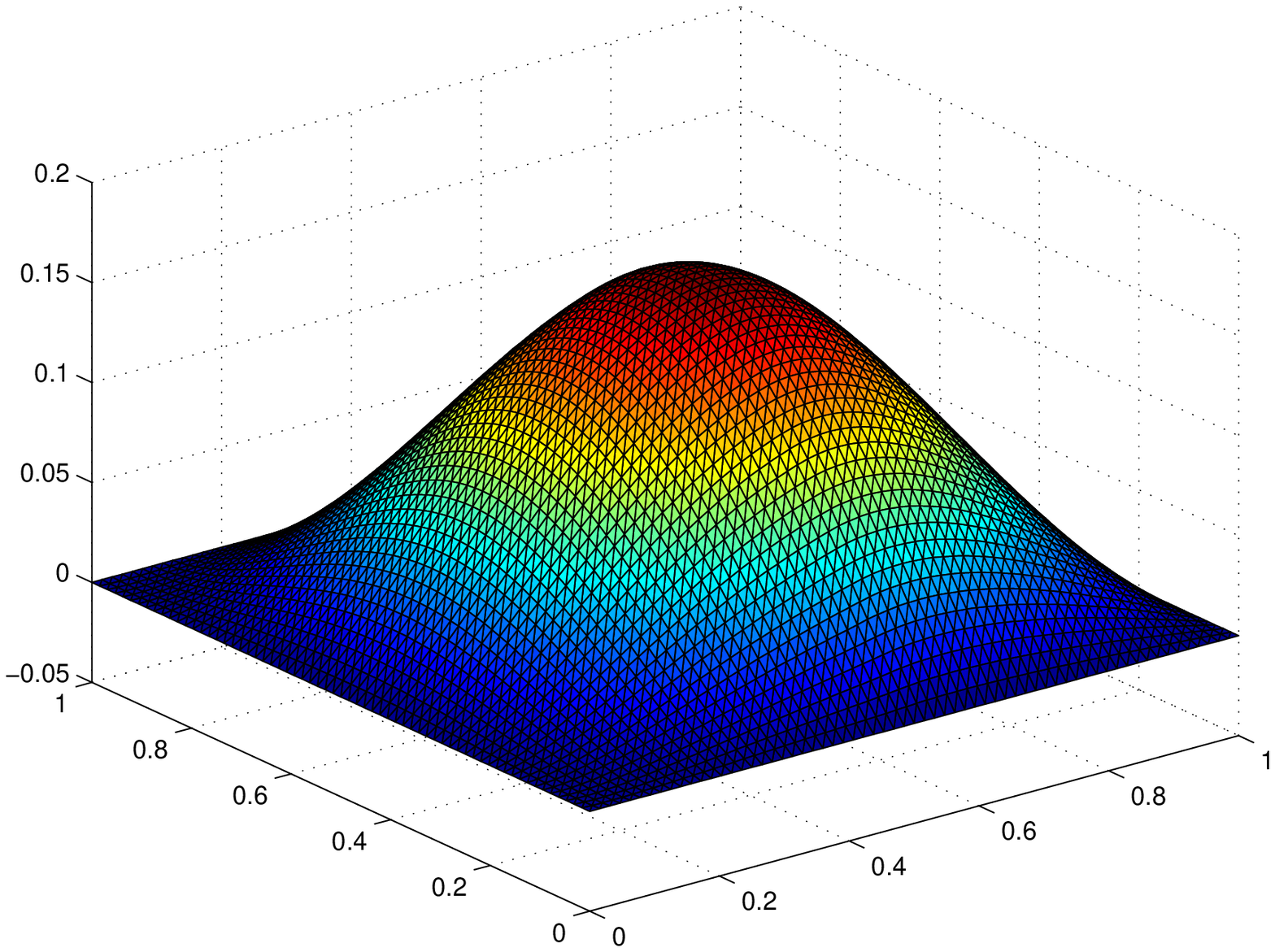}
        \label{fig:first_sub}
    }
    \subfigure[$\alpha = 0.5$]
    {
        \includegraphics[width=4.2cm]{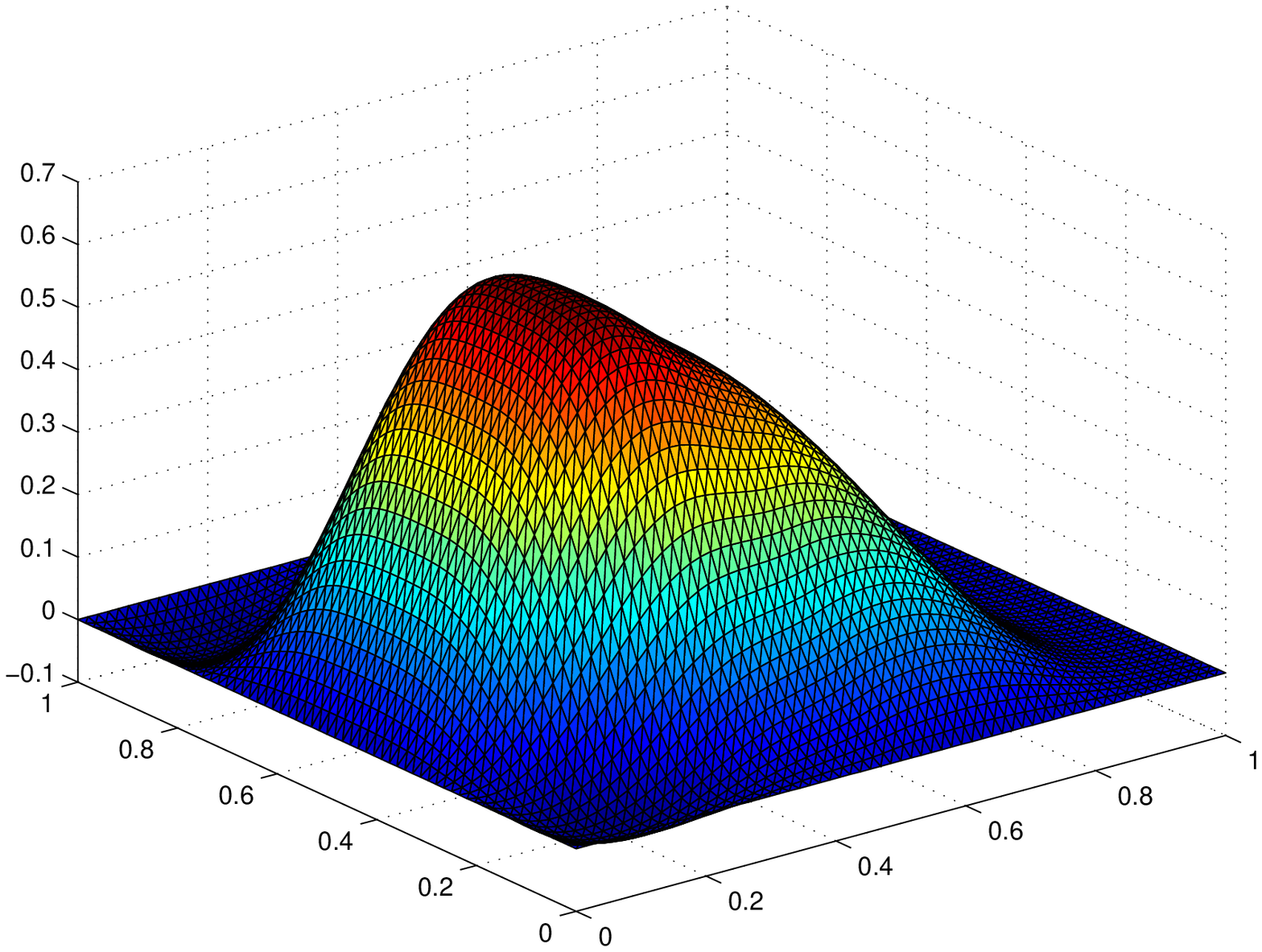}
        \label{fig:second_sub}
    }
    \subfigure[$\alpha = 0.9$]
    {
        \includegraphics[width=4.2cm]{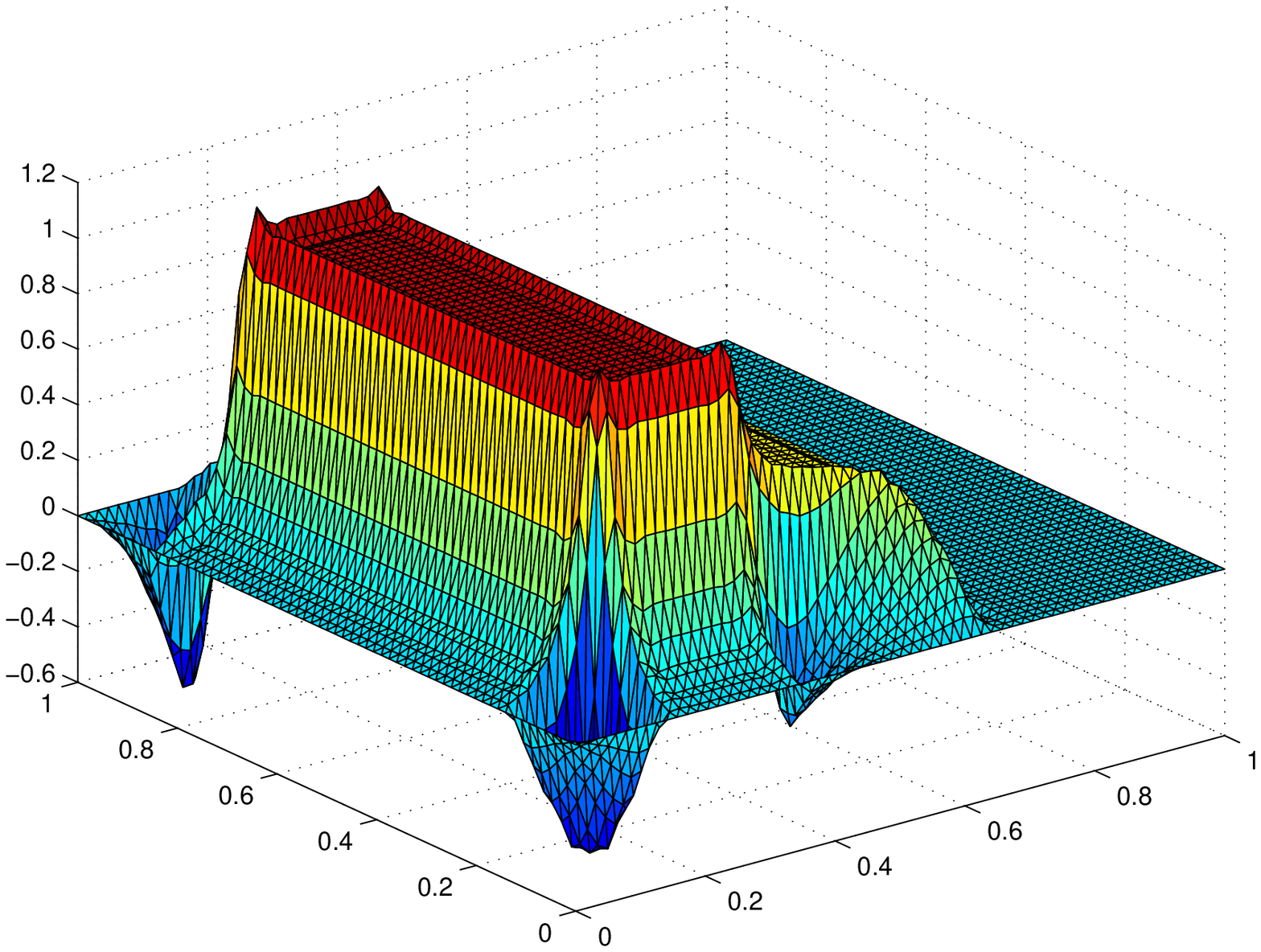}
        \label{fig:third_sub}
    }
    \caption{The profile of  solutions of problem \eqref{EE} at $t=0.1$ with different values of $\alpha$.}
    \label{fig:transition}
\end{figure}

For the wave-diffusion problem, the numerical results are listed in Table 6 for $\alpha = 0.5$.
We observe a $O(h^2)$ convergence  for  the $L^2(\Omega)$- and $L^\infty(\Omega)$-norm of the errors which confirms our predictions.  
It is known that the model \eqref{EE} interpolates the heat and wave equations
when  the fractional order $\alpha$ increases from zero to one. This transition is observed numerically.  
In Figure \ref{fig:transition}, we display the profile of the numerical solutions to case (c) at time $t=0.1$ with different values of $\alpha$. We  observe that, the closer $\alpha$ is to zero, the slower is the decay. Furthermore, the  oscillations in Figure \ref{fig:transition}(a) are inherited from the 
$L^2$-projection $P_hv$  which is oscillatory. This reflects, in particular, the wave feature of the 
model \eqref{EE}. 

\begin{table} 
\begin{center}
\caption{Numerical results for problem \eqref{EE},  $\alpha=0.5$, $\tau=1/500$.}
\label{table:6}
\begin{tabular}{|r|cccc|}
\hline
$M$& $L^2$-norm error & rate&$L^\infty$-norm error & rate \\
\hline
  8   & 5.7494e-003 &             & 1.0952e-002 &         \\
  16  & 1.4393e-003 & 2.00 & 2.7976e-003 & 1.97\\
  32  & 3.5725e-004 & 2.01 & 7.2567e-004 & 1.94\\
  64  & 8.5491e-005 & 2.06 & 1.9564e-004 & 1.89\\
  128 & 1.9769e-005 & 2.11 & 5.1351e-005 & 1.93\\
  \hline
\end{tabular}
\end{center}
\end{table}

	
\end{document}